\documentclass{article}
\usepackage{amsmath,amsfonts,enumitem,amsthm}

\usepackage[hmargin=3cm,vmargin=3cm]{geometry}

\usepackage{caption}
\usepackage{graphicx}
\usepackage{titling}
\pretitle{\noindent \centering \bfseries\LARGE}
\posttitle{\par\vskip 0.5em}
\preauthor{\noindent \centering \Large \lineskip 0.5em%
           \begin{tabular}[t]{c}}
\postauthor{\end{tabular}\par}

\usepackage{titlesec}
\titleformat{\section}
  {\normalfont\large\bfseries}{\thesection.}{1em}{}
\titleformat{\subsection}
  {\normalfont\large\bfseries}{\thesubsection}{1em}{}
\titleformat{\subsubsection}
  {\normalfont\normalsize\bfseries}{\thesubsubsection}{1em}{}
\titleformat{\paragraph}[runin]
  {\normalfont\normalsize\bfseries}{\theparagraph}{1em}{}
\titleformat{\subparagraph}[runin]
  {\normalfont\normalsize\bfseries}{\thesubparagraph}{1em}{}

\newtheorem{theorem}{Theorem}[section]

\newtheorem{corollary}[theorem]{Corollary}

\newtheorem{lemma}[theorem]{Lemma}

\newtheorem{proposition}[theorem]{Proposition}

\renewenvironment{proof}[1][Proof]{\noindent\textbf{#1.} }{\ \rule{0.5em}{0.5em}}

\theoremstyle{definition}
\newtheorem{definition}{Definition}[section]
\newtheorem{remark}[definition]{Remark}

\newcommand{\R}{\mathbb{R}}
\newcommand{\Z}{\mathbb{Z}}

\def\defn#1{{\bf\itshape #1}}

\usepackage{tikz}
\usetikzlibrary{shapes.geometric}
\usepackage{pgfplots}
\pgfplotsset{compat=1.13}
\usepackage{braids}
\usepackage{graphicx}

\definecolor{darktangerine}{rgb}{1.0, 0.66, 0.07}
\definecolor{iris}{rgb}{0.35, 0.31, 0.81}
\definecolor{asparagus}{rgb}{0.53, 0.66, 0.42}

\usepackage{tkz-euclide,amsmath} 
\usetkzobj{all} 
\tikzset{hidden/.style = {thick, dashed}}

\begin{document}
\title{Braids of the N-body problem by cabling a body in a central configuration}
\author{Marine Fontaine and Carlos Garc\'{\i}a-Azpeitia}
\date{}
\maketitle

\begin{center}
\begin{minipage}[]{14cm}
	\small \textbf{Abstract.} We prove the existence of periodic solutions of the $N=(n+1)$-body problem starting with $n$ bodies whose reduced motion is close to a non-degenerate central configuration and replacing one of them by the center of mass of a pair of bodies rotating uniformly. When the motion takes place in the standard
Euclidean plane, these solutions are a special type of braid solutions obtained numerically by C. Moore. The proof uses blow-up techniques to separate the problem into the $n$-body problem, the Kepler problem, and a coupling which is small if the distance of the pair is small. The formulation is variational and the result is obtained by applying a Lyapunov-Schmidt reduction and by using the equivariant Lyusternik-Schnirelmann category.

\vspace{0.2cm}

\vspace{0.2cm}
\end{minipage} 
\end{center}

\section{Introduction}

The discovery of braids and choreographies are linked since the appearance of
the original work \cite{Mo93} which contains the first choreography solution
differing from the classical Lagrange circular one. In this choreography,
three bodies follow one another along the now famous figure-eight orbit. The
result was obtained numerically by finding minimisers of the classical Euler
functional with a topological constraint associated with a braid. Later on,
the first rigorous mathematical proof of the existence of the figure-eight
orbit was obtained in \cite{ChMo00} by minimising the Euler functional over
paths that connect a collinear and an isosceles triangle configuration.
However, the name \emph{choreography} was adopted after the numerical work
\cite{Si00} to describe $n$ masses that follow the same path. The study of
choreographies has attracted much attention in recent years, while the study
of braids has been relatively less explored. The purpose of our paper is to
obtain new results on the existence of braids by cabling of central
configurations (Figure \ref{figure solution}).

\begin{figure}[t]
\captionsetup{width=.9\linewidth}
\par
\begin{center}%
\begin{tabular}
[c]{cc}%
\includegraphics[height=6cm]{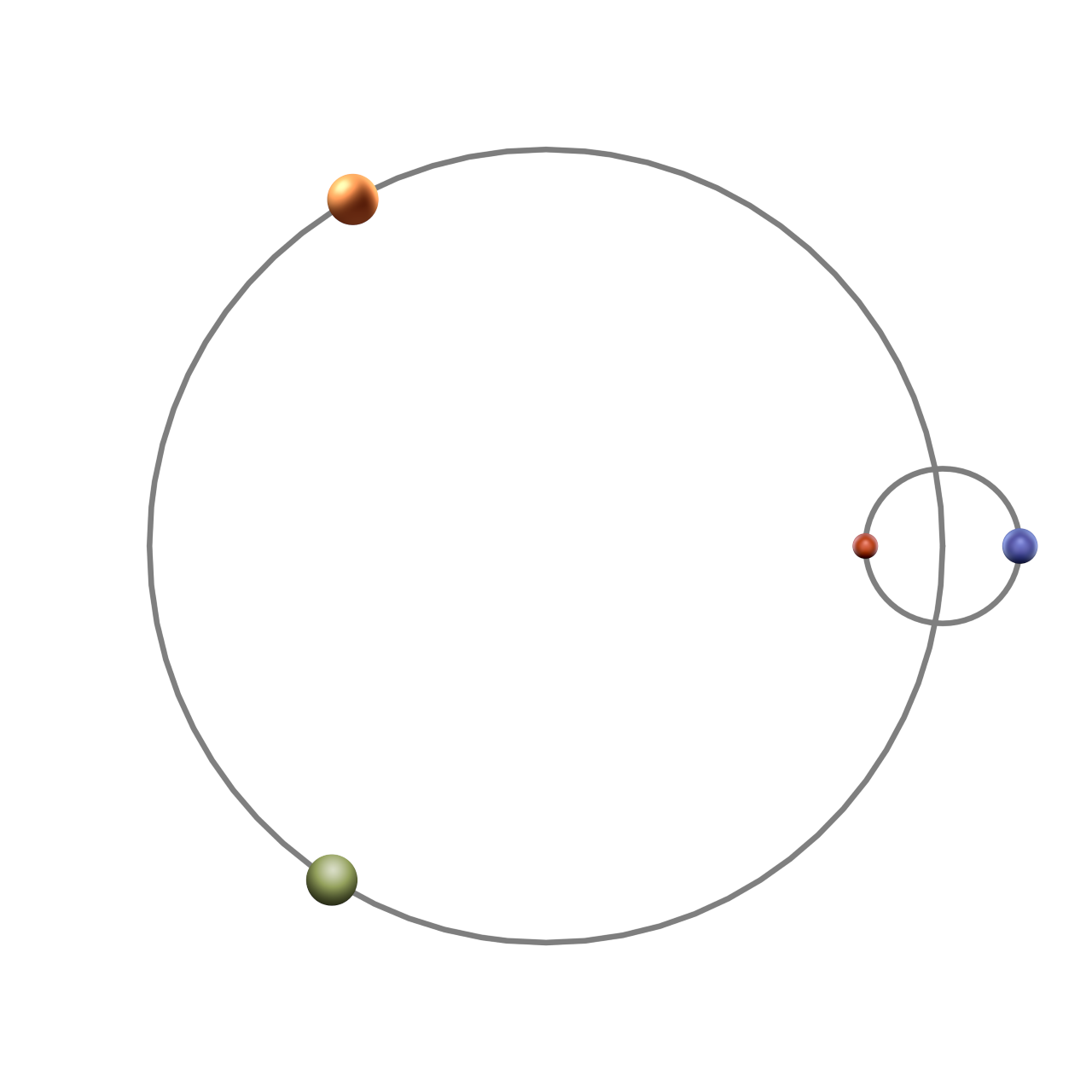} &
\includegraphics[height=6cm]{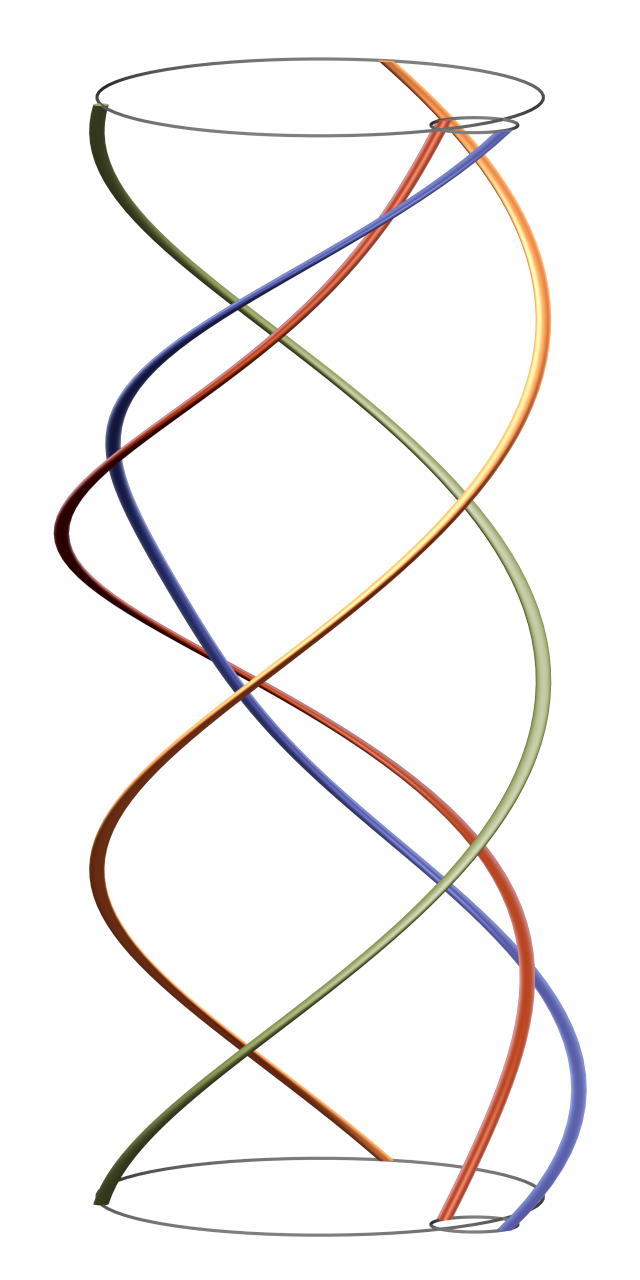}
\end{tabular}
\end{center}
\caption{ The left picture illustrates the orbit of a solution for $4$ bodies
in the plane ($d=1$). The right picture shows the same solution as a braid
solution for $4$ steady vortex filaments. The pair of bodies (red and blue)
wind around their center of mass two times while the other bodies (green and
yellow) and the center of mass of the pair wind around the origin one time. }%
\label{figure solution}%
\end{figure}

Concretely, we investigate the motion of $n$ bodies interacting under a
general homogeneous potential. The motion takes place in an even dimensional
Euclidean space $E$ equipped with a compatible complex structure $\mathcal{J}%
$. Denote by $Q_{\ell}(t)\in E$ the position of the $\ell$th body at time $t$
and let $M_{\ell}>0$ be its mass. Newton's equations are given by
\begin{equation}
M_{\ell}\ddot{Q}_{\ell}=-\sum_{k\neq\ell}M_{\ell}M_{k}\frac{Q_{\ell}-Q_{k}%
}{\left\Vert Q_{\ell}-Q_{k}\right\Vert ^{\alpha+1}}\text{,}\quad\ell=1,\dots,n
\label{body}%
\end{equation}
where $\alpha\geq1$. The case $\alpha=2$ corresponds to the problem of $n$
bodies moving under the influence of the gravitation. A central configuration
$a\in E^{n}$ is a configuration which gives rise to a solution of the form
$Q(t)=\exp(t\mathcal{J})a$. We construct braids of the $N=n+1$-body problem
starting with a central configuration $a$ of $n$ bodies. Without loss of
generality we may assume that $M_{1}=1$. The main idea is to replace one body
$Q_{1}$ by the center of mass of a pair of bodies $q_{0},q_{1}$ rotating
uniformly, with masses $m_{0},m_{1}>0$ such that $m_{0}+m_{1}=1$. We assume
that the central configuration $a$ is non-degenerate (definition
\ref{non degenerate cc}). The non-degeneracy of the Lagrange triangular
configuration and the Maxwell configuration (consisting of a central body and
$n$-bodies of equal masses attached to the vertices of a regular polygon)
follows a consequence of the stability analysis in
\cite{Mo92,MeSc93,Ro00,GaIz13} except for a finite number of mass parameters.

Specifically, our \textbf{main results} (Theorems \ref{main result} and
\ref{main result copy(1)}) state that, when the central configuration
$a=(a_{1},\dots,a_{n})$ is non-degenerate, there exists $\varepsilon_{0}>0$
such that, for all $\varepsilon\in(0,\varepsilon_{0})$, Newton's equations of
the $N=n+1$-body problem admit at least two solutions $q(t)=(q_{0}%
(t),\dots,q_{n}(t))$ such that%
\begin{align}
q_{0}(t) &  =\exp(t\mathcal{J})u_{1}(\nu t)-m_{1}\varepsilon\exp
(t\omega\mathcal{J})u_{0}(\nu t)\label{Nbody}\\
q_{1}(t) &  =\exp(t\mathcal{J})u_{1}(\nu t)+m_{0}\varepsilon\exp
(t\omega\mathcal{J})u_{0}(\nu t)\nonumber\\
q_{\ell}(t) &  =\exp(t\mathcal{J})u_{\ell}(\nu t),\qquad\ell=2,...,n,\nonumber
\end{align}
where the components $u_{\ell}=a_{\ell}+\mathcal{O}(\varepsilon)$ are $2\pi
$-periodic paths in $E$, $a_{0}\in E$ is a vector of unit length,
$\mathcal{O}(\varepsilon)$ is $2\pi$-periodic of order $\varepsilon$ with
respect to a Sobolev norm, and $\nu$ and $\omega$ are functions of
$\varepsilon$ through the relations $\omega=\pm\varepsilon^{-(\alpha+1)/2}$
and $\nu=\omega-1$. The sign of the frequency $\omega$ represents whether the
binary pair has a prograde or retrograde rotation. That is, \emph{prograde}
($\omega>0$) refers to the case that the pair rotates in the same direction as
the relative equilibrium, while \emph{retrograde} ($\omega<0$) refers to the
case that the pair rotates in the opposite direction. These solutions are
quasi-periodic if $\omega\notin\mathbb{Q}$, and periodic if $\omega
\in\mathbb{Q}$.

When $E$ is the plane and the frequency $\omega=\pm p/q$ is rational, there
is, for any fixed integer $q\in\mathbb{Z}\setminus\{0\}$, some $p_{0}>0$ such
that, for each $p>p_{0}$ , the components $q_{\ell}(t)$ of \eqref{Nbody} are
$2\pi q$-periodic. In these solutions $n-1$ bodies (close to $a_{\ell}$ for
$\ell=2,...,n$) and the center of mass of the pair $q_{0},q_{1}$ (close to
$a_{1}$) wind around the origin $q$ times, while the bodies $q_{0},q_{1}$ wind
around their center of mass $p$ times (see Corollary \ref{main braids} and
Figure \ref{figure solution}). These solutions are called \textbf{\itshape
braid solutions} in \cite{Mo93} and the process of replacing a body by a pair
is called \defn{cabling}. In the braid formalism this means replacing a strand
of a braid by another braid. For example, in Figure \ref{figure cabling}, the
rigid motion obtained by rotating the central configuration of three equal
masses located at the vertices of an equilateral triangle corresponds to the
braid $b_{1}$, and this motion is $2\pi$-periodic. Replacing one of the bodies
by the center of mass of two bodies rotating around their center of mass two
times after a complete period of $2\pi$ amounts to perform the cabling of the
braid $b_{1}$ with the braid of two strands $b_{2}$. The result is a new braid
$b_{1}\odot b_{2}$ with four strands.

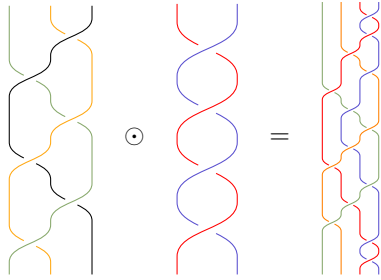
\begin{figure}[t]
\captionsetup{width=.9\linewidth}
\par
\begin{center}%
\begin{tabular}
[c]{ccccc}%
\begin{tikzpicture}[scale=0.55] \braid[style strands={1}{asparagus}, style strands={2}{darktangerine}, style strands={3}{black}] s_2^{-1} s_1^{-1}s_2^{-1} s_1^{-1}s_2^{-1} s_1^{-1}; \end{tikzpicture} &
\raisebox{50pt}{$\odot$} &
\begin{tikzpicture}[scale=0.8] \braid[style strands={1}{red}, style strands={2}{iris}, style strands={3}{green}] s_1^{-1}s_1^{-1}s_1^{-1}s_1^{-1}; \end{tikzpicture} &
\raisebox{50pt}{$=$} &
\begin{tikzpicture}[scale=0.25] \braid[style strands={1}{asparagus}, style strands={2}{orange}, style strands={3}{red}, style strands={4}{iris}] s_3^{-1}s_3^{-1}s_2^{-1}s_3^{-1}s_1^{-1}s_2^{-1}s_3^{-1}s_2^{-1}s_1^{-1}s_3^{-1}s_2^{-1}s_1^{-1}s_3^{-1}s_3^{-1}; \end{tikzpicture}\\
&  &  &  &
\end{tabular}
\end{center}
\caption{The picture illustrates the solution of Figure 1 as a braid. The
black strand in the braid $b_{1}$ on the left side is replaced by the braid
$b_{2}$ to form a new braid. The cabling operation is denoted by $b_{1}\odot
b_{2}$}%
\label{figure cabling}%
\end{figure}

For the case of the gravitational potential $\alpha=2$, the result for the $3
$-body problem ($N=2+1$) has been obtained separately by Moulton \cite{4} and
Siegel \cite{8}. They establish the existence of periodic solutions of the
$3$-body problem by combining two circular motions of the $2$-body problem.
This problem, which includes Hill's moon problem as a special case, enjoys a
large literature and has been treated from various point of views in the
original works \cite{2} by Hill and \cite{3} by Hopf. The case $N=3+1$ has
been studied in \cite{Ch}. The methods used in \cite{4,8} and \cite{Ch} to
prove the existence of solutions are quite different from ours.

Our method starts by writing the Euler-Lagrange equations with respect to the
Euler functional $\mathcal{A}$ of the $N$-body problem, with $N=n+1$. By
changing the variables in the configuration space, the Euler functional splits
into two terms $\mathcal{A=A}_{0}+\mathcal{H}$, where $\mathcal{A}_{0}$ is the
uncoupled Euler functional of the $n$-body problem and the Kepler problem. The
part $\mathcal{H}$ represents the interaction of the pair with the $n$-body
problem. Using the parameter $\varepsilon$, representing the radius of the
circular orbit of the Kepler problem, the coupling term $\mathcal{H=O}%
(\varepsilon)$ is small in order of $\varepsilon$, and the functional
$\mathcal{A}_{0}$ explodes as $\varepsilon\rightarrow0$.

If $\dim(E)=2d$, the functional $\mathcal{A}_{0}$ is invariant under the group $U(d)^{2}$
acting diagonally on the Kepler component $u_{0}\in E$ and the $n$ bodies
component $u\in E^{n}$, while the coupling term $\mathcal{H}$ is invariant
only by the action of the diagonal subgroup $\widetilde{U(d)}$ that rotates
the $N$-body problem. Let $x_{a}=(a_{0},a)$ where $a_{0}$ represents the
orientation of the circular orbit of the Kepler problem with respect to the
central configuration $a$. The $U(d)^{2}$-orbit of $x_{a}$ consists of
critical points of the unperturbed functional $\mathcal{A}_{0}$. In the
gravitational case $\alpha=2$, even if the central configuration $a$ is
non-degenerate, the group orbit of $x_{a}$ is degenerate due to the existence
of elliptic orbits. A similar problem arises when $E$ has at least dimension
four, due to resonances of the circular orbit of the Kepler problem with extra
dimensions. To deal with this issue, we need an extra assumption on the
symmetries of the central configuration $a$. Thus the functional $\mathcal{A}$
is invariant under the action of a discrete group $\Gamma$ and we can restrict
the study of critical points to the fixed point space of $\Gamma$. The
advantage is that, in the fixed point space of $\Gamma$, the problem of
resonances can be avoided.

The symmetry group of $\mathcal{A}_{0}$ will thus be taken to be a subgroup
$G=G_{1}\times G_{2}\subset U(d)\times U(d)\ $ such that it leaves the fixed
point space of $\Gamma$ invariant; similarly for the symmetry group
$H\subset\widetilde{U(d)}$ of the coupling term $\mathcal{H}$. Then the orbit
$G(x_{a})$ is non-degenerate in the space of periodic paths fixed by $\Gamma$
when $a$ is a non-degenerate central configuration. The core of the proof
(section $3$ and $4$) relies on several Lyapunov-Schmidt reductions in a
neighbourhood of $G(x_{a})$ such that one can solve the normal components to
the orbit $G_{2}(a)$. In this manner, finding critical points of $\mathcal{A}$
in a neighbourhood of $G(x_{a})$ is equivalent to finding critical $H_{a_{0}}%
$-orbits of the regular functional $\Psi_{\varepsilon}^{\prime}:G_{2}%
(a)\rightarrow\mathbb{R}$ defined on the compact manifold $G_{2}(a)$. The
delicate part of the proof consists in finding uniform estimates in
$\varepsilon$ because the functional $\mathcal{A}$ explodes when
$\varepsilon\rightarrow0$. The main theorem is obtained by computing the
$H_{a_{0}}$-equivariant Lyusternik-Schnirelmann category of the compact
manifold $G_{2}(a)$, which gives a lower bound for the number of $H_{a_{0}}%
$-orbits of critical points of $\Psi_{\varepsilon}^{\prime}$ along the lines
of \cite{Fo}.

Besides our interest in gravitational potentials ($\alpha=2$), we are
interested in the case $\alpha=1$ corresponding to solutions of steady
near-parallel vortex filaments in fluids. The equations for $\alpha=1$ govern
the interaction of steady vortex filaments in fluids (Euler equation)
\cite{Ne01}, Bose-Einstein condensates (Gross-Pitaevskii equation)
\cite{JerrardSmets16} and superconductors (Ginzburg-Landau equation)
\cite{ContrerasJerrard16}. Specifically, the positions of the steady
near-parallel vortex filaments are determined in space by%
\[
\left(  q_{j}(s),s\right)  \in\mathbb{C\times R\simeq R}^{3}\text{.}%
\]
Therefore, the solutions that we construct correspond to $N=n+1$ vortex
filaments forming helices, where one of the vortices is replaced by a pair of
vortices forming another helix (Figure \ref{figure solution}).

The existence of braids has been investigated previously under the assumption
that the force is strong (case $\alpha\geq3$) in \cite{Go1,Mon1} and
references therein. In the case of strong forces, the Euler functional blows
up at any orbit belonging to the boundary of a braid class because it contains
collisions. This allows to prove the existence of minimisers for most braid
classes by the direct method of calculus of variation for \emph{tied} braids
(which excludes the lack of coercitivity caused by the possibility that groups
of bodies escape to infinity). Similar results hold for the existence of
choreographic classes under the assumption of strong forces. In \cite{Mo} and
references therein the symmetry groups of choreographic classes have been
classified. A short exposition of different methods used to prove the
existence of choreographies can be found in \cite{Ga2019} and references therein.

However, the relevant cases from the physical point of view are the $N$-body
problem ($\alpha=2$) and the $N$-vortex filament problem ($\alpha=1$). The
difficulty to obtain minimisers on braid classes is that the minimiser of the
Euler functional may have collisions. In \cite{FeTe04} a method was developed
to obtain choreographies of the $N$-body problem ($\alpha=2$) as minimisers.
But finding braids of the $N$-body problem ($\alpha=2$) as minimisers is a
more difficult task. Furthermore, finding braids of the $N$-vortex filament
problem ($\alpha=1$) is more difficult than the $N$-body case ($\alpha=2$). In
this paper we propose a new method based on blow up methods (similar to
\cite{Ba16,Ba17}) to approach these problems. The blow-up method described in
this manuscript is part of a series of applications, namely (a) replacing one
body in a central configuration by $k$ bodies, (b) replacing each body in a
central configuration by $k_{j}$ bodies.

In section $2$ we set the problem of finding solutions of the $N$-body problem
arising as critical points of the Euler functional defined on a Sobolev space
and we discuss the symmetries of the problem. In section $3$ we perform a
Lyapunov-Schmidt reduction to a finite dimensional problem by using a
decomposition of paths in Fourier series. In section $4$ we perform a second
Lyapunov-Schmidt reduction to solve the normal components to the group orbit
and we obtain a lower bound for the critical points by using
Lyusternik-Schnirelmann methods. In section $5$ we discuss the existence of
braids (Theorem \ref{main result} and Corollary \ref{main braids}) by cabling
central configurations. We also discuss the solutions in higher dimension
(Theorem \ref{main result copy(1)}).

\paragraph{Acknowledgements.}

We acknowledge the assistance of Ramiro Chavez Tovar with the preparation of
the figures. We acknowledge the anonymous referees for their useful comments
which greatly improved the presentation of this article. M. Fontaine is
supported by the FWO-EoS project G0H4518N. C. Garc\'{\i}a-Azpeitia is
supported by PAPIIT-UNAM grant IN115019.

\section{Problem setting}

Let $E$ be a real Euclidean space with inner product $\langle\cdot
,\cdot\rangle$. Denote by $q:=(q_{0},q_{1},\dots,q_{n})\in E^{N}$ a
configuration of $N=n+1$ bodies in $E$ with masses $m_{0},\dots,m_{n}>0$. We
work only with configurations whose center of mass is fixed at the origin,
which amounts to say that the configuration space has been reduced by
translations. Define the kinetic energy and the potential function
\[
K=\frac{1}{2}\sum_{j=0}^{n}m_{j}\Vert\dot{q}_{j}\Vert^{2}\quad\mbox{and}\quad
U=\sum_{0\leq j<k\leq n}m_{j}m_{k}\phi_{\alpha}(\Vert q_{j}-q_{k}\Vert),
\]
where $\Vert\dot{q}_{j}\Vert^{2}=\langle\dot{q}_{j},\dot{q}_{j}\rangle$ and
$\phi_{\alpha}$ is a function such that $\phi_{\alpha}^{\prime}(r)=-r^{-\alpha
}$. The Newtonian potential corresponds to $\phi_{2}(r)=1/r $ and the vortex
filament potential corresponds to $\phi_{1}(r)=-\ln(r)$. The system of
equations of motion of the $N$-body problem reads
\[
m_{\ell}\ddot{q}_{\ell}=\nabla_{q_{\ell}}U=-\sum_{k\neq\ell}m_{\ell}m_{k}%
\frac{q_{\ell}-q_{k}}{\Vert q_{\ell}-q_{k}\Vert^{\alpha+1}},\qquad\ell
=0,\dots,n.
\]
Let $L=K+U$ be the Lagrangian of the system. The \defn{Euler functional}
\[
\mathcal{A}(q)=\int_{0}^{T}L((q(t),\dot{q}(t))dt
\]
is taken over the Sobolev space $H^{1}([0,T],E^{N})$ of paths
$q:[0,T]\rightarrow E^{N}$ such that $q$ and its first derivative $\dot{q}$
are square integrable in the sense of distributions.

\subsection{Jacobi-like coordinates}

Define fictional mass parameters $M_{0}=m_{0}m_{1}$, $M_{1}=m_{0}+m_{1}$ and
$M_{\ell}=m_{\ell}$ otherwise. After a rescaling we suppose that $M_{1}=1$.
Define new variables in the configuration space namely, $Q_{0}=q_{1}-q_{0}$,
$Q_{1}=m_{0}q_{0}+m_{1}q_{1}$, and $Q_{\ell}=q_{\ell}$ otherwise. Setting
$\mu_{0}=m_{1}$ and $\mu_{1}=-m_{0}$ we can write $q_{j}=Q_{1}-\mu_{j}Q_{0} $
for $j=0,1$. Observe that the center of mass of the configuration
\[
Q=(Q_{1},\dots,Q_{n}),
\]
with respect to the fictional masses $M_{1}, \dots, M_{n}$, remains at the origin.

\begin{proposition}
In the new coordinates $(Q_{0},Q)$, the kinetic energy and the potential
energy become
\[
K=\frac{1}{2}\sum_{j=0}^{n}M_{j}\Vert\dot{Q}_{j}\Vert^{2}\quad\mbox{and}\quad
U=M_{0}\phi_{\alpha}(\Vert Q_{0}\Vert)+\sum_{1\leq j<k\leq n}M_{j}M_{k}%
\phi_{\alpha}(\Vert Q_{j}-Q_{k}\Vert)+h(Q_{0},Q)
\]
with%
\begin{equation}
h(Q_{0},Q)=\sum_{k=2}^{n}\sum_{j=0,1}m_{k}m_{j}\left(  \phi_{\alpha}(\Vert
Q_{1}-\mu_{j}Q_{0}-Q_{k}\Vert)-\phi_{\alpha}(\Vert Q_{1}-Q_{k}\Vert)\right)
\text{.} \label{h in coordinates (Q_0,Q)}%
\end{equation}

\end{proposition}

\begin{proof}
Using that $m_{0}+m_{1}=1$, $q_{0}=Q_{1}-m_{1}Q_{0}$ and $q_{1}=Q_{1}%
+m_{0}Q_{0}$, we have
\[
\sum_{j=0,1}m_{j}\Vert\dot{q}_{j}\Vert^{2}=\Vert\dot{Q}_{1}\Vert^{2}+\left(
m_{0}m_{1}^{2}+m_{0}^{2}m_{1}\right)  \Vert\dot{Q}_{0}\Vert^{2}=M_{1}\Vert
\dot{Q}_{1}\Vert^{2}+M_{0}\Vert\dot{Q}_{0}\Vert^{2}.
\]
Then $K=\frac{1}{2}\sum_{j=0}^{n}M_{j}\Vert\dot{Q}_{j}\Vert^{2}$. For the
potential energy we have%
\begin{align*}
U  &  =\sum_{j<k}m_{j}m_{k}\phi_{\alpha}(\left\Vert q_{j}-q_{k}\right\Vert )\\
&  =m_{0}m_{1}\phi_{\alpha}(\left\Vert q_{0}-q_{1}\right\Vert )+\sum_{k=2}%
^{n}\sum_{j=0,1}m_{k}m_{j}\phi_{\alpha}(\left\Vert q_{j}-q_{k}\right\Vert
)+\sum_{2\leq j<k\leq n}m_{j}m_{k}\phi_{\alpha}(\left\Vert q_{j}%
-q_{k}\right\Vert )\\
&  =M_{0}\phi_{\alpha}(\left\Vert Q_{0}\right\Vert )+\sum_{1\leq j<k\leq
n}M_{j}M_{k}\phi_{\alpha}(\left\Vert Q_{j}-Q_{k}\right\Vert )+h(Q_{0},Q),
\end{align*}
where%
\[
h(Q_{0},Q)=\sum_{k=2}^{n}\sum_{j=0,1}m_{k}m_{j}\phi_{\alpha}(\left\Vert
q_{j}-Q_{k}\right\Vert )-\sum_{k=2}^{n}M_{1}M_{k}\phi_{\alpha}(\left\Vert
Q_{1}-Q_{k}\right\Vert )\text{.}%
\]
Since $M_{1}=m_{0}+m_{1}=1$, and $q_{k}=Q_{k}$ and $m_{k}=M_{k}$ for $k\geq2
$, we obtain%
\[
h(Q_{0},Q)=\sum_{k=2}^{n}\sum_{j=0,1}m_{k}m_{j}\left(  \phi_{\alpha
}(\left\Vert q_{j}-Q_{k}\right\Vert )-\phi_{\alpha}(\left\Vert Q_{1}%
-Q_{k}\right\Vert )\right)  .
\]
The result for $h$ follows from the fact that $q_{j}=Q_{1}-\mu_{j}Q_{0}$ for
$j=0,1$.
\end{proof}

The Euler functional splits into two terms
\begin{equation}
\mathcal{A}(Q_{0},Q)=\mathcal{A}_{0}(Q_{0},Q)+\mathcal{H}(Q_{0},Q).
\label{euler functional}%
\end{equation}
They are explicitly given by
\[
\mathcal{A}_{0}(Q_{0},Q)=\int_{0}^{T}\frac{1}{2}\sum_{j=0}^{n}M_{j}\Vert
\dot{Q}_{j}(t)\Vert^{2}+M_{0}\phi_{\alpha}(\Vert Q_{0}(t)\Vert)+\sum_{1\leq
j<k\leq n}M_{j}M_{k}\phi_{\alpha}(\Vert Q_{j}(t)-Q_{k}(t)\Vert)\,dt
\]
and $\mathcal{H}(Q_{0},Q)=\int_{0}^{T}h(Q_{0}(t),Q(t))\,dt$ with $h$ as in
\eqref{h in coordinates (Q_0,Q)}. Notice that $h(Q_{0},Q)$ is an analytic
function in a neighbourhood of $Q_{0}=0$ with $h(Q_{0},Q)=\mathcal{O}%
(\left\Vert Q_{0}\right\Vert )$. Furthermore $h$ is invariant under linear
isometries
\begin{equation}
h(gQ_{0},gQ)=h(Q_{0},Q) \label{invariance of h}%
\end{equation}
where $g\in O(E)$ and $gQ=(gQ_{1},\dots,gQ_{n})$.

\subsection{Rotating-like coordinates}

Since we already reduced the space by translations, a relative equilibrium of
the $n$-body problem is now a solution of the Newton's equations which is an
equilibrium after reducing the configuration space by the group of linear
isometries $O(E)$ acting diagonally on $E^{n}$. That is, the motion is of the
form $Q(t)=\exp(t\Lambda)a$ for a fixed configuration $a\in E^{n}$ and a
skew-symmetric matrix $\Lambda$. Since $\Lambda$ is non-degenerate on the
space of motion (see \cite{dist}), we may suppose from the beginning that $E$
is even dimensional and is endowed with a compatible almost complex structure.
We set $\dim(E)=2d$ and pick a basis such that the complex structure is block
diagonal%
\[
\mathcal{J}:=J\oplus\dots\oplus J\text{,}%
\]
where $J$ is the standard symplectic matrix on $\mathbb{R}^{2}$. We define
rotating-like coordinates
\[
Q_{j}(t)=\exp(t\mathcal{J})v_{j}(t).
\]
In the coordinates $v_{j}$, the two terms of the Euler functional
\eqref{euler functional} become
\[
\mathcal{A}_{0}(v_{0},v)=\int_{0}^{T}\frac{1}{2}\sum_{j=0}^{n}M_{j}%
\Vert\left(  \partial_{t}+\mathcal{J}\right)  v_{j}(t)\Vert^{2}+\,M_{0}%
\phi_{\alpha}(\Vert v_{0}(t)\Vert)+\sum_{1\leq j<k\leq n}M_{j}M_{k}%
\phi_{\alpha}(\Vert v_{j}(t)-v_{k}(t)\Vert)\,dt
\]
and $\mathcal{H}(v_{0},v)=\int_{0}^{T}h(v_{0}(t),v(t))\,dt$ remains unchanged
because of its invariance under linear isometries \eqref{invariance of h}. The
Euler-Lagrange equations for $\mathcal{A}_{0}$ are
\begin{align}
\frac{\delta\mathcal{A}_{0}}{\delta v_{0}}  &  =-M_{0}\left(  \partial
_{t}+\mathcal{J}\right)  ^{2}v_{0}-M_{0}\frac{v_{0}}{\left\Vert v_{0}%
\right\Vert ^{\alpha+1}}=0\label{v}\\
\frac{\delta\mathcal{A}_{0}}{\delta v_{\ell}}  &  =-M_{\ell}\left(
\partial_{t}+\mathcal{J}\right)  ^{2}v_{\ell}\,-\sum_{k=1(k\neq\ell)}%
^{n}M_{\ell}M_{k}\frac{v_{\ell}-v_{k}}{\left\Vert v_{\ell}-v_{k}\right\Vert
^{\alpha+1}}=0\text{.} \label{u}%
\end{align}

Equation \eqref{v} is the Kepler problem in rotating coordinates. Equations
\eqref{u} are Newton's equations for $n$ bodies with masses $M_{1},\dots
,M_{n}$ in rotating coordinates. A \textbf{\itshape central configuration}
$a=(a_{1},\dots,a_{n})\in E^{n}$ satisfies the equations%

\[
a_{\ell}=\sum_{k\neq\ell}M_{k}\frac{a_{\ell}-a_{k}}{\left\Vert a_{\ell}%
-a_{k}\right\Vert ^{\alpha+1}}.
\]
Therefore, $a$ is an equilibrium of equations \eqref{u}, and the motion
$Q(t)=\exp(t\mathcal{J})a$ is a relative equilibrium. Central configurations
can also be defined as critical points of the \defn{amended potential} of the
$n$-body problem
\begin{equation}
\label{amended potential V}V(v)=\frac{1}{2}\sum_{j=1}^{n}M_{j}\left\Vert
v_{j}\right\Vert ^{2}+\sum_{1\leq k<j\leq n}M_{j}M_{k}\phi_{\alpha}(\left\Vert
v_{j}-v_{k}\right\Vert ).
\end{equation}
Then $a\in E^{n}$ is a central configuration if and only if $\nabla V(a)=0$.

\subsection{Time and space scaling}

Equation (\ref{v}) is the Kepler problem for homogeneous potentials in
rotating coordinates. This equation has solutions corresponding to circular
orbits. We consider a special type of circular orbits of the form
\[
v_{0}(t)=\varepsilon\exp(\left(  \omega-1\right)  \mathcal{J}t)a_{0},
\]
where%
\[
a_{0}\in E\quad\mbox{is of unit length and}\quad\omega=\pm\varepsilon
^{-(\alpha+1)/2}.
\]
The case $\omega>0$ corresponds to a prograde rotation of the pair and
$\omega<0$ to a retrograde rotation.

We introduce a change of coordinates which is particularly useful to continue
the circular solution of (\ref{v}) and the equilibrium of \eqref{u}. This
change of coordinates is defined by
\begin{align*}
v_{0}(t)  &  =\varepsilon\exp(\left(  \omega-1\right)  \mathcal{J}t)u_{0}(\nu
t)\\
v_{\ell}(t)  &  =u_{\ell}(\nu t),\qquad\ell=1,...,n
\end{align*}
where $\nu\in\R$ is a frequency. We shall now introduce a new time parameter
$s=\nu t$ and write
\[
x(s)=(u_{0}(s),u(s)).
\]
In the proposition below, we express the Euler functional $\mathcal{A}%
(v_{0},v)$ in terms of the new coordinates $(u_{0},u)$. This functional is
referred to as a \textit{normalised} functional because it involves a scaling
in time of the old functional by a factor $\nu$. By making an abuse of
notation it is still denoted $\mathcal{A}(x)=\mathcal{A}_{0}(x)+\mathcal{H}(x)$.
Note that the old and the new functionals have the same critical points.
Moreover, for any central configuration $a\in E^{n}$ and any unit length
vector $a_{0}\in E$, the constant path
\begin{equation}
x_{a}(s)=(a_{0},a)
\end{equation}
is a critical point of the unperturbed functional $\mathcal{A}_{0}(x)$.

\begin{proposition}
Suppose $\alpha\geq1$. In the coordinates $x(s)=(u_{0}(s),u(s))$, the
normalised Euler functional $\mathcal{A}(x)=\mathcal{A}_{0}(x)+\mathcal{H}(x)$
is given by the two terms
\begin{align}
\mathcal{A}_{0}(x)  &  =\varepsilon^{1-\alpha}M_{0}\int_{0}^{2\pi}\frac{1}%
{2}\Vert\left(  \frac{\nu}{\omega}\partial_{s}+\mathcal{J}\right)
u_{0}(s)\Vert^{2}+\phi_{\alpha}(\Vert u_{0}(s)\Vert
)~ds\nonumber\label{h in term of u}\\
&  +\int_{0}^{2\pi}\frac{1}{2}\sum_{j=1}^{n}M_{j}\Vert\left(  \nu\partial
_{s}+\mathcal{J}\right)  u_{j}(s)\Vert^{2}+\sum_{1\leq j<k\leq n}M_{j}%
M_{k}\phi_{\alpha}(\Vert u_{j}(s)-u_{k}(s)\Vert)~ds\nonumber\\
\mathcal{H}(x)  &  =\int_{0}^{2\pi}h\left(  \varepsilon\exp\left(
\frac{\omega-1}{\nu}s\,\mathcal{J}\right)  u_{0}(s),u(s)\right)  ~ds~.
\end{align}
Explicitly the integrand $h$ is
\[
h(u_{0}, u)=\sum_{k=1}^{n}\sum_{j=0,1}M_{k}m_{j}\left(  \phi_{\alpha}(\Vert
u_{1}(s)-\mu_{j}\varepsilon\exp\left(  \frac{\omega-1}{\nu}s\,\mathcal{J}%
\right)  u_{0}(s)-u_{k}(s)\Vert)-\phi_{\alpha}(\Vert\left(  u_{1}%
(s)-u_{k}(s)\Vert\right)  \right)  .
\]

\end{proposition}

\begin{proof}
When $\alpha>1$ the potential $\phi_{\alpha}$ is homogeneous of degree
$1-\alpha$, then
\[
\phi_{\alpha}(\Vert v_{0}(t)\Vert)=\varepsilon^{1-\alpha}\phi_{\alpha}(\Vert
u_{0}(s)\Vert).
\]
Moreover
\[
\left\Vert \left(  \partial_{t}+\mathcal{J}\right)  v_{0}(t)\right\Vert
^{2}=\left\Vert \varepsilon\left(  \nu\partial_{s}+\omega\mathcal{J}\right)
u_{0}(s)\right\Vert ^{2}=\varepsilon^{1-\alpha}\left\Vert \left(  \frac{\nu
}{\omega}\partial_{s}+\mathcal{J}\right)  u_{0}(s)\right\Vert ^{2}%
\]
and the result follows by rescaling $\mathcal{A}$ by $\nu$.

The case $\alpha=1$ is similar, but now $\phi_{\alpha}(\Vert v_{0}%
(t)\Vert)=\phi_{\alpha}(\Vert u_{0}(s)\Vert)-\ln(\varepsilon)$ and
\[
\left\Vert \left(  \partial_{t}+\mathcal{J}\right)  v_{0}(t)\right\Vert
^{2}=\left\Vert \left(  \frac{\nu}{\omega}\partial_{s}+\mathcal{J}\right)
u_{0}(s)\right\Vert ^{2}.
\]
The result for $\alpha=1$ follows by rescaling $\mathcal{A}$ by $\nu$ and
adding the constant $-2\pi M_{0}\ln(\varepsilon)$.

Finally the term $h$ is given by changing the coordinates in
\eqref{h in coordinates (Q_0,Q)}. Since the term is invariant by rotations we
may replace each term $Q_{k}$ by $v_{k}$. In terms of the coordinates
$(u_{0},u)$ it becomes
\[
h=\sum_{k=1}^{n}\sum_{j=0,1}M_{k}m_{j}\left(  \phi_{\alpha}(\Vert u_{1}(\nu
t)-\mu_{j}\varepsilon\exp\left(  (\omega-1)t\mathcal{J}\right)  u_{0}(\nu
t)-u_{k}(\nu t)\Vert)-\phi_{\alpha}(\Vert\left(  u_{1}(\nu t)-u_{k}(\nu
t)\Vert\right)  \right)  .
\]
By changing the time parameter as $s=\nu t$, we obtain the expression written
in the statement.
\end{proof}

\subsection{Gradient formulation}

Let $\mathbb{S}^{1}=\R/2\pi\Z$ and consider the subset of $2\pi$-periodic
paths,
\[
X:=H^{1}(\mathbb{S}^{1},E^{N})\subset H^{1}([0,2\pi],E^{N}).
\]
The space $X$ is a real Hilbert space with inner product
\[
(x_{1},x_{2})_{X}=(x_{1},x_{2})_{L^{2}}+(\dot{x}_{1},\dot{x}_{2})_{L^{2}}%
=\int_{0}^{2\pi}\langle x_{1}(s),x_{2}(s)\rangle+\langle\dot{x}_{1}(s),\dot
{x}_{2}(s)\rangle ds.
\]
The topological dual $X^{\prime}$ is identified with the the Sobolev space of
distributions $X^{-1}$ defined by%
\[
X^{s}:=H^{s}(\mathbb{S}^{1},E^{N})=\left\{  (\hat{x}_{\ell})_{\ell
\in\mathbb{Z}}\mid\sum_{\ell\in\mathbb{Z}}(\ell^{2}+1)^{s}\Vert\hat{x}_{\ell
}\Vert^{2}<\infty\right\}  ,
\]
where $(\hat{x}_{\ell})$ is the sequence of Fourier coefficients in
$(E_{\mathbb{C}})^{N}=(E\oplus iE)^{N}$ of $x$ satisfying $\hat{x}_{\ell
}=\overline{\hat{x}}_{-\ell}$. In particular, an element $x\in X^{s}$ for
$s\geq0$ can be written as a Fourier series for the function $x(s)=\sum
_{\ell\in\mathbb{Z}}\hat{x}_{\ell}e^{i\ell s}$. On the other hand, an element
$y\in X^{-s}$ for $s>0$ is a distribution that acts on a test function
$x(s)=\sum_{\ell\in\mathbb{Z}}\hat{x}_{\ell}e^{i\ell s}\in X^{s}$ by the
formula $(y,x)=\sum_{\ell\in\mathbb{Z}}\hat{y}_{\ell}\cdot\hat{x}_{\ell}.$

For a given open collision-less subset $\Omega\subset X$ we denote by
$d\mathcal{A}:\Omega\subset X\rightarrow X^{\prime}$ the differential of
$\mathcal{A}$. Using the identification between the dual space $X^{\prime}$
and $X^{-1}$ we define the operator of first variation $\delta\mathcal{A}%
:\Omega\subset X\rightarrow X^{-1}$ satisfying $d\mathcal{A}(x)(y)=(\delta
\mathcal{A}(x),y)$. On the other hand, by the Riesz representation theorem,
the gradient operator $\nabla\mathcal{A}:\Omega\subset X\rightarrow X$ is
uniquely defined by $d\mathcal{A}(x)(y)=(\nabla\mathcal{A}(x),y)_{X}$. Using
an integration by parts and the fact that the paths are periodic, we obtain
that
\[
(\delta\mathcal{A}(x),y)=d\mathcal{A}(x)(y)=(\nabla\mathcal{A}(x),y)_{X}%
=((-\partial_{s}^{2}+1)\nabla\mathcal{A}(x),y)
\]
where $(-\partial_{s}^{2}+1)^{-1}:X^{-1}\rightarrow X$ is the Riesz map. Thus
we conclude that
\[
\nabla\mathcal{A}=(-\partial_{s}^{2}+1)^{-1}\delta\mathcal{A}:\Omega\subset
X\rightarrow X\text{.}%
\]

For $x\in X$ the Euler-Lagrange equations of the unperturbed functional
$\mathcal{A}_{0}$ in gradient formulation are
\begin{align}
\nabla_{u_{0}}\mathcal{A}_{0}(x)  &  =\left(  -\partial_{s}^{2}+1\right)
^{-1}\varepsilon^{1-\alpha}M_{0}\left(  -\left(  \frac{\nu}{\omega}%
\partial_{s}+\mathcal{J}\right)  ^{2}u_{0}-\frac{u_{0}}{\left\Vert
u_{0}\right\Vert ^{\alpha+1}}\right)  =0\label{gradient A_0 u}\\
\nabla_{u_{\ell}}\mathcal{A}_{0}(x)  &  =\left(  -\partial_{s}^{2}+1\right)
^{-1}M_{\ell}\left(  -\left(  \nu\partial_{s}+\mathcal{J}\right)  ^{2}u_{\ell
}-\sum_{k\neq\ell}M_{k}\frac{u_{\ell}-u_{k}}{\left\Vert u_{\ell}%
-u_{k}\right\Vert ^{\alpha+1}}\right)  =0\text{.}%
\end{align}
Note that the operators $\left(  \frac{\nu}{\omega}\partial_{s}+\mathcal{J}%
\right)  ^{2}$ and $\left(  \nu\partial_{s}+\mathcal{J}\right)  ^{2}$ are
defined from $X$ to $X^{-1}$ whereas $\left(  -\partial_{s}^{2}+1\right)
^{-1}$ is defined from $X^{-1}$ to $X$. Furthermore, it is possible to
consider the composition of these operators as operators from $X$ to $X$
without passing by the dual space $X^{-1}$. For instance, given $x\in X$,%
\begin{equation}
\left(  -\partial_{s}^{2}+1\right)  ^{-1}\left(  \nu\partial_{s}%
+\mathcal{J}\right)  ^{2}x=\sum_{\ell\in\mathbb{Z}}\frac{1}{\ell^{2}+1}\left(
i\ell\nu\mathcal{I}+\mathcal{J}\right)  ^{2}\hat{x}_{\ell}e^{i\ell s}\text{.}
\label{OpFou}%
\end{equation}

The above equations admit the solution path $x_{a}\in X$ given by
\begin{equation}
x_{a}(s)=(a_{0},a),\quad\forall s\in\mathbb{S}^{1}. \label{x_a}%
\end{equation}
We want to prove that there are critical solutions $x(s)=(u_{0}(s),u(s))$
close to $x_{a}$ that persist as critical solutions for the perturbed
functional $\mathcal{A}(x)=\mathcal{A}_{0}(x)+\mathcal{H(}x)$.

To ensure that $\exp\left(  \frac{\omega-1}{\nu}\mathcal{J}s\right)  $ is
$2\pi$-periodic, and $\mathcal{H}$ is well defined in $X$, we need to impose
the condition $\omega=1+m\nu$ for some $m\in\Z$. In particular, we imposed the
following conditions on the set of parameters:

\begin{description}
\item[(A)] $\omega=\pm\varepsilon^{-(\alpha+1)/2}$.

\item[(B)] $\omega=1+\nu$.
\end{description}

We will prove that $\nabla\mathcal{H}(x)=\mathcal{O}_{X}(\varepsilon)$ in the
space $X$. Condition \textbf{(A)} implies that $x_{a}$ is a critical point of
$\mathcal{A}_{0}(x)$ and condition \textbf{(B)} that $\mathcal{H}(x)$ is well
defined in the space of $2\pi$-periodic paths $X$. The critical solutions of
$\mathcal{A}(x)$ provide solutions of the $N$-body problem. We prove the
existence of a continuum of solutions when $\varepsilon\rightarrow0$.
Conditions \textbf{(A)}-\textbf{(B)} determine $\omega$ and $\nu$ as functions
of $\varepsilon$ such that $\omega,\nu\rightarrow+\infty$ when $\varepsilon
\rightarrow0$ for the prograde rotation, and $\omega,\nu\rightarrow-\infty$
for the retrograde rotation. In principle, we do not need to assume that the
parameter $\omega$ is rational. Braids are particular solutions such that
$d=1$ and $\omega\in\mathbb{Q}$.

\subsection{Discrete and continuous symmetries}

\label{section: symmetries}

Since $U(d)$ is the centraliser of $\mathcal{J}$ in $O(E)$, the unperturbed
functional $\mathcal{A}_{0}$ is invariant with respect to the product group
$U(d)\times U(d)$. The first factor acting on the component $u_{0}$, and the
second factor acting diagonally on the $n$ last components $u\in E^{n}$. The
action of this group extends on $X$ by rotating non simultaneously the Kepler
orbit and the central configuration; that is,
\[
(g_{1},g_{2})(u_{0},u)=(g_{1}u_{0},g_{2}u),\qquad(g_{1},g_{2})\in U(d)\times
U(d)
\]
where $g_{2}u=(g_{2}u_{1},\dots,g_{2}u_{n})$. Observe that the coupling term
$\mathcal{H}$ in the functional breaks the symmetry of $\mathcal{A}_{0}$ and
the perturbed functional $\mathcal{A=A}_{0}+\mathcal{H}$ is only invariant
with respect to the diagonal subgroup
\[
\widetilde{U(d)}=\{(g_{1},g_{1})\in U(d)\times U(d)\mid g_{1}\in U(d)\},
\]
acting by rotating the $N=n+1$ bodies with respect to the origin. We now
distinguish the three following cases:

\begin{itemize}
\item[\textbf{(C1)}] $E$ is the plane $(d=1)$ and $\alpha\neq2$.

\item[\textbf{(C2)}] $E$ is the plane $(d=1)$ and $\alpha=2$ (Newtonian case).

\item[\textbf{(C3)}] $E$ is of higher dimension $(d\geq2)$ and $\alpha\geq1$.
\end{itemize}

Those cases need to be treated separately in Lemma \ref{Estimate hessian A_0}
in order to perform a reduction of dimension. Indeed, the reduction relies on
the invertibility of a regularised hessian operator at the critical point on
some slice in $X$. The invertibility fails in cases $\mathbf{(C2)}$ and
$\mathbf{(C3)}$. In case $\mathbf{(C2)}$ this is due to the appearance of
resonances given by elliptic orbits, and in case $\mathbf{(C3)}$ this is due
to the presence of resonances in higher dimension. To deal with this issue, we
make use of an extra discrete symmetry subgroup $\Gamma$ of the perturbed
functional $\mathcal{A}$. The problem of resonances can be avoided when
working on the fixed point space $X^{\Gamma}$ instead of $X$. This is allowed
by the principle of symmetric criticality of Palais \cite{Palaisa}. In this
case $x_{a}$ needs to be chosen such that $x_{a}$ $\in X^{\Gamma}$ and,
similarly, the symmetry group $G$ of $\mathcal{A}_{0}$ must be chosen so that
it leaves $X^{\Gamma}$ invariant. Note that there may be other solutions
outside of this fixed point space. We discuss below which discrete symmetry is
relevant for each case and which symmetry group $G$ must be taken. The
discrete symmetry also restricts the type of central configurations we can
braid, at least in the case $\mathbf{(C2)}$ and $\mathbf{(C3)}$.

\paragraph{(C1)}

No restriction is needed in this case, there are no resonances. We may take
$\Gamma$ to be the trivial group, $G=U(1)\times U(1)$ and $H=\widetilde{U(1)}%
$. We then study the critical points of $\mathcal{A}$ in $X^{\Gamma}=X $.

\paragraph{(C2)}

The bodies are now moving in the plane under the influence of the Newtonian
gravitational force. We can braid the central body of symmetric central
configurations which include symmetric configurations at the origin, the
Maxwell configuration and nested polygonal configurations with a center (see
section \ref{examples}). For each case, we can find a discrete symmetry group
$\Gamma$ that allows to deal with the resonances.

Let $S_{n}$ be the permutation group of $n$ letters and consider the discrete
subgroup $\Gamma<\mathbb{Z}_{m}\times S_{n}$ generated by a non-trivial
element $(\theta,\sigma)$ such that
\[
\theta=2\pi/m\in\mathbb{Z}_{m},\qquad\sigma^{m}=(1)\in S_{n},\qquad
\sigma(1)=1.
\]
This group acts on $X$ as follows: for $x\in X$ we have
\[
(\theta,\sigma)x(s)=(u_{0}(s+\theta),\exp(-\theta\mathcal{J})u_{\sigma
(1)}(s+\theta),\dots,\exp(-\theta\mathcal{J})u_{\sigma(n)}(s+\theta)).
\]

\textbf{(C2a)} The first assumption on the central configuration is that the
masses satisfy
\begin{equation}
M_{\ell}=M_{\sigma(\ell)}.
\end{equation}
The functional $\mathcal{A}_{0}$ is $\Gamma$-invariant because, in its
expression, the variables $u_{0}(s)$ and $u_{\ell}(s)$ are uncoupled.
Furthermore, in the next proposition we show that the coupling term
$\mathcal{H}$ is $\Gamma$-invariant. Thus the functional $\mathcal{A}$ is
$\Gamma$-invariant and we can restrict the study of its critical points to the
fixed point set $X^{\Gamma}$.

\begin{proposition}
Under condition \textbf{(C2a) }the action functional $\mathcal{H}(x)$ is
$\Gamma$-invariant
\end{proposition}

\begin{proof}
Consider the term $h$ given in \eqref{h in term of u} with $\left(
\omega-1\right)  /\nu=1$. Given $(\theta,\sigma)\in\Gamma$, we first write
explicitly $h\left(  (\theta,\sigma)(\varepsilon\exp\left(  s\mathcal{J}%
\right)  u_{0}(s),u(s))\right)  $. Using \eqref{h in term of u} this gives
\begin{align*}
&  \sum_{k=2}^{n}\sum_{j=0,1}M_{k}m_{j}\phi_{\alpha}(\Vert\exp(-\theta
\mathcal{J})u_{\sigma(1)}(s+\theta)-\mu_{j}\varepsilon\exp(s\mathcal{J}%
)u_{0}(s+\theta)-\exp(-\theta\mathcal{J})u_{\sigma(k)}(s+\theta)\Vert)\\
&  -\sum_{k=2}^{n}\sum_{j=0,1}M_{k}m_{j}\phi_{\alpha}(\Vert\exp(-\theta
\mathcal{J})u_{\sigma(1)}(s+\theta)-\exp(-\theta\mathcal{J})u_{\sigma
(k)}(s+\theta)\Vert).
\end{align*}
One can now use the $SO(2)$-invariance of the norms and $\sigma(1)=1$ to
rewrite this as
\begin{align*}
&  \sum_{k=2}^{n}\sum_{j=0,1}M_{k}m_{j}\phi_{\alpha}(\Vert u_{1}(s+\theta
)-\mu_{j}\varepsilon\exp((s+\theta)\mathcal{J})u_{0}(s+\theta)-u_{\sigma
(k)}(s+\theta)\Vert)\\
&  -\sum_{k=2}^{n}\sum_{j=0,1}M_{k}m_{j}\phi_{\alpha}(\Vert u_{1}%
(s+\theta)-u_{\sigma(k)}(s+\theta)\Vert).
\end{align*}
In particular, changing the indices in the summation with respect to $k$
implies that
\[
h\left(  (\theta,\sigma)(\varepsilon\exp\left(  s\mathcal{J}\right)
u_{0}(s),u(s))\right)  =h\left(  \varepsilon\exp\left(  (s+\theta
)\mathcal{J}\right)  u_{0}(s+\theta),u(s+\theta)\right)  .
\]
Since $\mathcal{H}$ is defined in the space of $2\pi$-periodic functions, we
get
\begin{align*}
\mathcal{H}((\theta,\sigma)x) &  =\int_{0}^{2\pi}h\left(  \varepsilon
\exp\left(  (s+\theta)\mathcal{J}\right)  u_{0}(s+\theta),u(s+\theta)\right)
~ds\\
&  =\int_{\theta}^{2\pi+\theta}h\left(  \varepsilon\exp(s^{\prime}%
\mathcal{J})u_{0}(s^{\prime}),u(s^{\prime})\right)  ~ds^{\prime}%
=\mathcal{H}(x)\text{.}%
\end{align*}

\end{proof}

\textbf{(C2b)} The second assumption (to ensure that $x_{a}\in X^{\Gamma}$) is
that the central configuration $a\in E^{n}$ satisfies the property
\begin{equation}
a_{\ell}=\exp(-\theta\mathcal{J})a_{\sigma(\ell)}.
\end{equation}
Since $\sigma^{m}=(1)$ and $\theta=2\pi/m$, conditions \textbf{(C2a)-(C2b)} imply that
the central configuration $a$ is symmetric by $2\pi/m$-rotations in the plane
and that $a_{1}=0$.

Symmetric configurations that satisfy this condition are discussed in section
\ref{examples}. In this case the group action of $U(1)\times U(1)$ on $X$
commutes with the action of $\Gamma$, then we can take $G=U(1)\times U(1)$ and
$H=\widetilde{U(1)}$.

\paragraph{(C3)}

We now consider the higher dimensional case; that is when the space of motion
$E$ is at least four dimensional. Let $\Gamma$ be the finite subgroup
isomorphic to $\mathbb{Z}_{2}$ whose generator $\zeta$ acts on $X$ as
follows:
\[
\zeta x(s)=(-\mathcal{R}u_{0}(s+\pi),\mathcal{R}u(s+\pi))\text{,}%
\]
where
\[
\mathcal{R}=-I_{2}\oplus I_{2}\oplus...\oplus I_{2}\in\mbox{End}(E)\text{.}%
\]

The functional $\mathcal{A}_{0}$ is $\Gamma$-invariant because $\mathcal{R}$
commutes with $\mathcal{J}$. Similarly, the functional $\mathcal{H}$ is
invariant because
\begin{align*}
\mathcal{H}(\zeta x)  &  =\int_{0}^{2\pi}h\left(  -\varepsilon\exp\left(
\left(  s-\pi\right)  \mathcal{J}\right)  \mathcal{R}u_{0}(s),\mathcal{R}%
u_{1}(s),...,\mathcal{R}u_{n}(s)\right)  ~ds\\
&  =\int_{0}^{2\pi}h\left(  \mathcal{R}\varepsilon\exp\left(  s\mathcal{J}%
\right)  u_{0}(s),\mathcal{R}u_{1}(s),...,\mathcal{R}u_{n}(s)\right)
~ds=\mathcal{H}(x)\text{.}%
\end{align*}
Therefore, the functional $\mathcal{A}$ is $\Gamma$-invariant and we can
restrict the study of critical points to the fixed point space $X^{\Gamma}$.
In this case we choose the symmetry group $G$ to be the maximal subgroup of
$U(d)\times U(d)$ acting on $X^{\Gamma}$. The groups are thus of the form
$G=G_{1}\times G_{2}$ and $H=\widetilde{G_{1}}$, where each $G_{i}$ is the
centraliser of $\mathcal{R}$ in $U(d)$; that is
\[
G_{i}=U(1)\times U(d-1).
\]

Note that $x_{a}=(a_{0},a)\in X^{\Gamma}$ if and only if $-\mathcal{R}%
a_{0}=a_{0}$ and $\mathcal{R}a_{j}=a_{j}$ for $j=1,..,n$. Therefore $x_{a}$
must be taken such that $a_{0}$ lies in the plane
\[
\Pi=\{(x,y,0,...,0)\}\subset E,
\]
and the central configuration $a$ consists of points lying in the orthogonal
complement $\Pi^{\perp}\subset E$. The choice of symmetry group $G$ ensures
that $G(x_{a})\subset X^{\Gamma}$. In dimension four $(d=2)$ the Kepler orbit
is located in a plane and the central configuration lies in an orthogonal plane.

Choosing the symmetry group $G$ and the path $x_{a}$ accordingly to one of the
assumptions $\mathbf{(C1)}$ or $\mathbf{(C2),(C3)}$, the equations
\eqref{gradient A_0 u} vanish along the orbit $G(x_{a})$ and the real question
to answer is whether some orbits of solutions along the orbit persist in the
space $X^{\Gamma}$ when considering the perturbation term $\mathcal{H}%
=\mathcal{O}(\varepsilon)$ for small $\varepsilon$. For this purpose, we
suppose that the collision-less neighbourhood $\Omega$ is of the form
\[
\Omega=\{x\in X\mid\exists g\in U(d)\times U(d),~~\left\Vert x-gx_{a}%
\right\Vert _{X}<\rho\}
\]
for some $\rho>0$. Then in further applications, one shall take an open subset
$\Omega^{\Gamma}\subset\Omega\cap X^{\Gamma}$ which is a $\rho$-neighbourhood
of radius $\rho$ around the group orbit $G(x_{a})$,

\begin{proposition}
The functional $\mathcal{A=A}_{0}+\mathcal{H}$ is well defined in
$\Omega\subset X$.
\end{proposition}

\begin{proof}
Since $\left\Vert x\right\Vert _{C^{0}}\leq\gamma\left\Vert x\right\Vert _{X}$
by Sobolev embedding, the paths $x\in\Omega$ do not leave the pointwise
neighbourhood of the orbit
\[
\widetilde{\Omega}=\{y\in E^{N}\mid\exists g\in U(d)\times U(d),~~\left\Vert
y-gx_{a}\right\Vert _{E^{N}}<\gamma\rho\}.
\]
The potential energy\ $U$ and the nonlinear term $h$ are pointwise analytic
functions defined in $\widetilde{\Omega}$ if $\rho$ is small enough. Since
paths in $\Omega$ do not leave $\widetilde{\Omega}$, i.e. $x\in\Omega$ implies
$x(s)\in\widetilde{\Omega}$ for all $s\in\mathbb{S}^{1}$, the Euler functional
$\mathcal{A}$ is well defined in the region $\Omega\subset X$ if $\rho$ is
small enough.
\end{proof}

Hereafter, we use the Banach algebra property of $X$ and the analyticity of
$\mathcal{A}$ to obtain functional estimates of its derivatives. In
particular, we have the following estimate:

\begin{lemma}
\label{lemma: bound for gradient of H} There is a constant $N_{2}>0$ such that
the compact operator $\nabla\mathcal{H}:X\to X$ satisfies
\[
\left\Vert \nabla\mathcal{H}(x)\right\Vert _{X}\leq N_{2}\varepsilon\text{.}%
\]
uniformly for $x\in\Omega$ and small $\varepsilon$.
\end{lemma}

\begin{proof}
After setting $\frac{\omega-1}{\nu}=1$, the
integrand term in $\mathcal{H}(x)=\int_{0}^{2\pi}h(\varepsilon\exp
(s\mathcal{J})u_{0}(s),u(s))ds$ is
\[
h(\varepsilon\exp(s\mathcal{J})u_{0}(s),u(s))=\sum_{k=2}^{n}\sum_{j=0,1}%
M_{k}m_{j}\left(  \phi_{\alpha}(\Vert u_{1}(s)-\mu_{j}\varepsilon\exp\left(
\mathcal{J}s\right)  u_{0}(s)-u_{k}(s)\Vert)-\phi_{\alpha}(\Vert\left(
u_{1}(s)-u_{k}(s)\Vert\right)  \right)  
\]
by using \eqref{h in term of u}. A straightforward calculation yields
\[
\frac{\delta\mathcal{H}}{\delta u_{0}}=\varepsilon\sum_{k=2}^{n}\sum
_{j=0,1}M_{k}m_{j}\mu_{j}\exp(\mathcal{J}s)^{t}\left(  \frac{u_{1}%
(s)-u_{k}(s)-\mu_{j}\varepsilon\exp(\mathcal{J}s)u_{0}(s)}{\Vert
u_{1}(s)-u_{k}(s)-\mu_{j}\varepsilon\exp\left(  \mathcal{J}s\right)
u_{0}(s)\Vert^{\alpha+1}}\right)  .
\]
Notice that the term
\begin{equation}
\frac{u_{1}-u_{k}-\mu_{j}\varepsilon u_{0}}{\Vert u_{1}-u_{k}-\mu
_{j}\varepsilon u_{0}\Vert^{\alpha+1}}\label{c1}%
\end{equation}
is real analytic for $x=(u_{0},...,u_{n})\in\widetilde{\Omega}$ and small
$\varepsilon$, i.e. it satisfies%
\[
\left\Vert \frac{u_{1}-u_{k}-\mu_{j}\varepsilon u_{0}}{\Vert u_{1}-u_{k}%
-\mu_{j}\varepsilon u_{0}\Vert^{\alpha+1}}\right\Vert _{E}\leq C_{E},\qquad
x\in\widetilde{\Omega}\text{.}%
\]
Notice that $\left(  \exp(\mathcal{J}s)u_{0}(s),u_{1}(s),...,u_{n}(s)\right)
\in\widetilde{\Omega}$ pointwise for any $x\in\Omega$ by the embedding
$X\subset C^{0}$. By the Banach algebra property of $X$, we conclude that%
\[
\left\Vert \frac{u_{1}(s)-u_{k}(s)-\mu_{j}\varepsilon\exp(\mathcal{J}%
s)u_{0}(s)}{\Vert u_{1}(s)-u_{k}(s)-\mu_{j}\varepsilon\exp\left(
\mathcal{J}s\right)  u_{0}(s)\Vert^{\alpha+1}}\right\Vert _{X}\leq
C_{X},\qquad x\in\Omega\text{.}%
\]
Therefore, we have that $\nabla_{u_{0}}\mathcal{H}(x)=\left(  -\partial
_{s}^{2}+1\right)  ^{-1}\frac{\delta\mathcal{H}}{\delta u_{0}}$ is a compact
operator of order $\varepsilon$. That is, $\left\Vert \nabla_{u_{0}%
}\mathcal{H}(x)\right\Vert _{X}\leq N_{2}\varepsilon$ with the constant
$N_{2}$ independent of $x\in\Omega.$

Similarly, one obtains
\[
\frac{\delta\mathcal{H}}{\delta u_{1}}=-\sum_{k=2}^{n}\sum_{j=0,1}M_{k}%
m_{j}\left(  \frac{u_{1}(s)-\mu_{j}\varepsilon\exp\left(  \mathcal{J}s\right)
u_{0}(s)-u_{k}(s)}{\Vert u_{1}(s)-\mu_{j}\varepsilon\exp\left(  \mathcal{J}%
s\right)  u_{0}(s)-u_{k}(s)\Vert^{\alpha+1}}-\frac{u_{1}(s)-u_{k}(s)}{\Vert
u_{1}(s)-u_{k}(s)\Vert^{\alpha+1}}\right)  .
\]
Notice that the function%
\begin{equation}
\frac{u_{1}-\mu_{j}\varepsilon u_{0}-u_{k}}{\Vert u_{1}-\mu_{j}\varepsilon
u_{0}-u_{k}\Vert^{\alpha+1}}-\frac{u_{1}-u_{k}}{\Vert u_{1}-u_{k}\Vert
^{\alpha+1}}=\mathcal{O}_{E}(\varepsilon)\label{c2}%
\end{equation}
is real analytic for $x=(u_{0},...,u_{n})\in\widetilde{\Omega}$ and its Taylor
expansion with respect to $\varepsilon$ has vanishing constant term. We
conclude by a similar argument that $\left\Vert \nabla_{u_{1}}\mathcal{H}%
(x)\right\Vert _{X}\leq N_{2}\varepsilon$ with the constant $N_{2}$
independent of $x\in\Omega$. The result follows by noticing that $\frac
{\delta\mathcal{H}}{\delta u_{k}}=-\frac{\delta\mathcal{H}}{\delta u_{1}}$ for
$k=2,...,n$.
\end{proof}

\begin{remark}
\label{hessian} The functions (\ref{c1}) and (\ref{c2}) are real analytic for
$x\in\widetilde{\Omega}$ and small $\varepsilon$. By the Banach algebra
property of $X$, all the successive derivatives of $\nabla\mathcal{H}%
(x)=\left(  -\partial_{s}^{2}+1\right)  ^{-1}\delta\mathcal{H}(x)$ are bounded
operators with operator norms of order $\varepsilon$ for all $x\in\Omega$. In
particular, the operator norm of $\nabla^{2}\mathcal{H}(x):X\rightarrow X$ is
of order $\varepsilon$, $\left\Vert \nabla^{2}\mathcal{H}(x)\right\Vert
\leq C\varepsilon$ for $x\in\Omega$. Actually, the operator $\mathcal{A}$
and its components $\mathcal{A}_{0}$ and $\mathcal{H}$ are analytic
functionals in the domain $\Omega\subset X$ in the sense of definition 2.3.1
in \cite{Be77}.
\end{remark}

\section{Lyapunov-Schmidt reduction}

As before, we take the standard parametrisation $\mathbb{S}^{1}=\mathbb{R}%
/2\pi\mathbb{Z}$ and we identify
\[
X=H^{1}(\mathbb{S}^{1},E^{N})=\left\{  x\in L^{2}(\mathbb{S}^{1},E^{N}%
)\mid\sum_{\ell\in\mathbb{Z}}(\ell^{2}+1)\Vert\hat{x}_{\ell}\Vert^{2}%
<\infty\right\}  ,
\]
where $(\hat{x}_{\ell})$ is the sequence of Fourier coefficients in
$(E_{\mathbb{C}})^{N}=(E\oplus iE)^{N}$ satisfying $\hat{x}_{\ell}%
=\overline{\hat{x}}_{-\ell}$. Write an element $x\in X$ as a Fourier series
$x=\sum_{\ell\in\mathbb{Z}}\hat{x}_{\ell}e_{\ell}$ where $e_{\ell}%
:\mathbb{S}^{1}\rightarrow\mathbb{C}$ is given by $e_{\ell}(s)=e^{i\ell s}$.
Then we can write $X=X_{0}\oplus W$, where $X_{0}$ is the subspace of constant
loops and $W$ is the subspace of loops in $X$ having zero mean. Thus any
element $x\in X$ decomposes uniquely as $x=\xi+\eta$, where%
\[
\xi=\hat{x}_{0},\qquad\eta=\sum_{\ell\neq0}\hat{x}_{\ell}e_{\ell}.
\]
Denote by $P:X\rightarrow X_{0}$ the canonical projection onto $X_{0}$, then
$Px=\xi$ and $(I-P)x=\eta$, where $I$ denotes the identity on $X$. The system
of equations $\nabla\mathcal{A}(\xi+\eta)=0$ splits into
\begin{align*}
\nabla_{\xi}\mathcal{A}(\xi+\eta)  &  =P\nabla\mathcal{A}(\xi+\eta)=0\in
X_{0}~,\\
\nabla_{\eta}\mathcal{A}(\xi+\eta)  &  =(I-P)\nabla\mathcal{A}(\xi+\eta)=0\in
W.
\end{align*}

Reducing the system to finite dimension by mean of the Lyapunov-Schmidt
reduction requires to solve the equation $\nabla_{\eta}\mathcal{A}(\xi
+\eta)=0$. For this purpose, we define $F_{\varepsilon}:\Omega\subset
X\rightarrow W$ as the operator
\[
F_{\varepsilon}(\xi,\eta):=\mathcal{D}_{\varepsilon}\nabla_{\eta}%
\mathcal{A}(\xi+\eta),
\]
where $\mathcal{D}_{\varepsilon}\in\mbox{End}(E^{N})$ is the block diagonal
matrix
\begin{equation}
\mathcal{D}_{\varepsilon}=\varepsilon^{\alpha-1}\mathcal{I}\oplus
\varepsilon^{\alpha+1}\mathcal{I}\oplus\dots\oplus\varepsilon^{\alpha
+1}\mathcal{I},
\end{equation}
where $\mathcal{I}$ denotes the identity on $E$. Solving the second equation
is equivalent to solving $F_{\varepsilon}(\xi,\eta)=0$ for $\varepsilon\neq0$
because $\mathcal{D}_{\varepsilon}$ is an isomorphism. While $\nabla_{\eta
}\mathcal{A}(\xi+\eta)$ explodes as $\varepsilon\rightarrow0$, the function
$F_{\varepsilon}(\xi,\eta)$ is continuous at $\varepsilon=0$ because
$\lim_{\varepsilon\rightarrow0}\left(  \nu/\omega\right)  ^{2}=1$. Therefore,%
\[
F_{0}(\xi,\eta)=\lim_{\varepsilon\rightarrow0}\left(  \mathcal{D}%
_{\varepsilon}\nabla_{\eta}\mathcal{A}_{0}(\xi+\eta)\right)
\]
is well defined. Furthermore, $F_{0}(gx_{a},0)=0$ for all $g\in G$. Solving
$F_{\varepsilon}(\xi,\eta)=0$ requires the functional derivative
$\partial_{\eta}F_{0}[(gx_{a},0)]$ to be invertible on $W$. Although this is
true when working under condition $\mathbf{(C1)}$, the operator is not
invertible on the whole space $W$ under condition $\mathbf{(C2)-(C3)}$ (see
the lemma below). However, in those bad cases, the operator is invertible on
$W^{\Gamma}$.

We use the notations $X_{0}^{\Gamma}$ and $W^{\Gamma}$ to denote the
projections of $X^{\Gamma}=(X_{0}\oplus W)^{\Gamma}$ on the first and second
factor, respectively.

\begin{lemma}
\label{Estimate hessian A_0} Assume conditions $\mathbf{(A)}-\mathbf{(B)}$.
Under assumption $\mathbf{(C1)}$, the operator $\partial_{\eta}F_{0}%
[(gx_{a},0)]$ is invertible on $W$ for all $g\in G$, i.e. there is a constant
$c>0$ such that
\[
\Vert\partial_{\eta}F_{0}[(gx_{a},0)]^{-1}\eta\Vert\leq c\Vert\eta\Vert
\quad\mbox{for every}\quad\eta\in W,~g\in G.
\]
Under assumptions $\mathbf{(C2)}$ or $\mathbf{(C3)}$, the same result holds
when the operator $\partial_{\eta}F_{0}[(gx_{a},0)]$ is restricted to the
fixed point space $W^{\Gamma}$, with $\Gamma$ and $G$ chosen accordingly to
those assumptions.
\end{lemma}

\begin{proof}
We first write the Hessian of $\mathcal{A}_{0}$ at $x_{a}$ as the block
diagonal matrix
\[
\nabla^{2}\mathcal{A}_{0}[x_{a}]=\nabla_{u_{0}}^{2}\mathcal{A}_{0}%
[x_{a}]\oplus\nabla_{u}^{2}\mathcal{A}_{0}[x_{a}].
\]
A straightforward calculation yields
\[
\nabla_{u_{0}}^{2}\mathcal{A}_{0}[x_{a}]=\left(  -\partial_{s}^{2}+1\right)
^{-1}M_{0}\varepsilon^{1-\alpha}\left(  -(\nu/\omega)^{2}\mathcal{I}%
\partial_{s}^{2}-2(\nu/\omega)\mathcal{J}\partial_{s}+(\alpha+1)a_{0}a_{0}%
^{t}\right)  ,
\]
where $a_{0}^{t}$ denotes the transpose of $a_{0}$. Similarly,
\[
\nabla_{u}^{2}\mathcal{A}_{0}[x_{a}]=\left(  -\partial_{s}^{2}+1\right)
^{-1}\left(  -\nu^{2}\mathcal{M}\partial_{s}^{2}-2\nu\mathcal{M}%
\mathcal{J}_{n}\partial_{s}+\nabla^{2}V[a]\right)  ,
\]
where $\mathcal{M}=M_{1}\mathcal{I}\oplus\dots\oplus M_{n}\mathcal{I}$ and
$\mathcal{J}_{n}=\mathcal{J}\oplus\dots\oplus\mathcal{J}$ are block diagonal
matrices, both with $n$ blocks of size $2d$.

Let $\eta=\sum_{\ell\neq0}\hat{x}_{\ell}e_{\ell}\in W$ and write
\[
\partial_{\eta}F_{0}[(x_{a},0)]\eta=\sum_{\ell\neq0}\hat{T}_{\ell}\hat
{x}_{\ell}e_{\ell}%
\]
where the matrix $\hat{T}_{\ell}$ is block diagonal of the form
\begin{equation}
\hat{T}_{\ell}=\hat{T}_{\ell,u_{0}}\oplus\hat{T}_{\ell,u}. \label{G_0 blocks}%
\end{equation}
Since the coefficients $\hat{x}_{\ell}$ do not depend on $s$ we get
\[
\partial_{s}\eta=\sum_{\ell\neq0}i\ell\hat{x}_{\ell}e_{\ell}\quad
\mbox{and}\quad\partial_{s}^{2}\eta=-\sum_{\ell\neq0}\ell^{2}\hat{x}_{\ell
}e_{\ell}.
\]
Since $\lim_{\varepsilon\rightarrow0}\left(  \nu/\omega\right)  =1$, the first
block in \eqref{G_0 blocks} is given by
\[
\hat{T}_{\ell,u_{0}}=\frac{M_{0}}{\ell^{2}+1}\left(  (\ell^{2}\mathcal{I}%
-2i\ell\mathcal{J}+(\alpha+1)a_{0}a_{0}^{t}\right)  .
\]
Without loss of generality, suppose $a_{0}=(1,0,\dots,0)\in E$. Hence the
block $\hat{T}_{\ell,u_{0}}$ is diagonal of the form%
\begin{equation}
\hat{T}_{\ell,u_{0}}=\frac{M_{0}}{\ell^{2}+1}\left(  \left(
\begin{array}
[c]{cc}%
\ell^{2}+(\alpha+1) & -2i\ell\\
2i\ell & \ell^{2}%
\end{array}
\right)  \bigoplus_{d-1}\left(
\begin{array}
[c]{cc}%
\ell^{2} & -2i\ell\\
2i\ell & \ell^{2}%
\end{array}
\right)  \right)  . \label{block}%
\end{equation}
The matrix $\hat{T}_{\ell,u_{0}}$ has eigenvalues%
\begin{equation}
\lambda_{1,\ell}^{\pm}=\frac{M_{0}}{\ell^{2}+1}\left(  \ell^{2}+\frac
{\alpha+1}{2}\pm\frac{1}{2}\sqrt{16\ell^{2}+(\alpha+1)^{2}}\right)  ,
\label{eigval1}%
\end{equation}
which appear with multiplicity one, and%
\[
\lambda_{2,\ell}^{\pm}=\frac{M_{0}}{\ell^{2}+1}\ell\left(  \ell\pm2\right)
\]
which appear with multiplicity $d-1$. We now study the invertibility for each
assumption $\mathbf{(C1)}$, $\mathbf{(C2)}$ and $\mathbf{(C3)}$.

\textbf{(C1) }Since we are working on the plane, the matrix $\hat{T}%
_{\ell,u_{0}}$ has only the two eigenvalues \eqref{eigval1}. Since $\ell\neq0$
and $\alpha\neq2$, these eigenvalues never vanish. This proves invertibility.

\textbf{(C2) }In Fourier components, $x$ is fixed by $\Gamma<\mathbb{Z}%
_{m}\times S_{n}$ if and only if $x(s)=(\theta,\sigma)x(s)$. This enforces
$u_{0}$ to be $2\pi/m$-periodic. Therefore, the Fourier expansion of $u_{0}$
is fixed by $\Gamma$ only if
\begin{equation}
u_{0,\ell}=0\text{ for }\ell\neq0,\pm m,\pm2m,... \label{zero condition1}%
\end{equation}
Since the eigenvalues of the matrix $\hat{T}_{\ell,u_{0}}$ are not singular
for $\alpha=2$ as long as $\ell\neq\pm1$, then the operator $\partial_{\eta
}F_{0}[(x_{a},0)]$ restricted to $W^{\Gamma}$ is invertible.

\textbf{(C3)} In Fourier components, $x$ is fixed by $\Gamma$ if an only if
\[
\sum_{\ell\in\mathbb{Z}}(u_{0,\ell},u_{\ell})e^{i\ell s}=x(s)=\zeta
x(s)=\sum_{\ell\in\mathbb{Z}}(-\mathcal{R}u_{0,\ell},\mathcal{R}u_{\ell
})e^{i\left(  \ell s+\pi\ell\right)  }.
\]
Set%
\[
u_{0,\ell}=u_{0,\ell}^{1}\oplus u_{0,\ell}^{2}.
\]
This implies that $\hat{x}_{\ell}=(u_{0,\ell},u_{\ell})\in E^{N}$ is fixed by
$\Gamma$ only if
\begin{align}
u_{0,\ell}^{1}  &  =0\text{ for }\ell\neq0,\pm2,\pm
4,...\label{zero condition2}\\
u_{0,\ell}^{2}  &  =0\text{ for }\ell\neq\pm1,\pm3,\pm5,...\nonumber
\end{align}
Since the eigenvalues of the matrix $\hat{T}_{\ell,u_{0}}$ for the component
$u_{0,\ell}^{1}$ are non zero as long as $\ell\neq\pm1$ (in the case
$\alpha=2$) and for the component $u_{0,\ell}^{2}$ if $\ell\neq\pm2$, then the
operator $\partial_{\eta}F_{0}[(x_{a},0)]$ restricted to $W^{\Gamma}$ is invertible.

The limits of the eigenvalues of $\hat{T}_{\ell,u_{0}}$ tends to $M_{0}$ when
$\ell\rightarrow\infty$. Since $\lim_{\varepsilon\rightarrow0}\left(
\varepsilon^{\alpha+1}\nu^{2}\right)  =1$, the second block in
\eqref{G_0 blocks} is
\[
\hat{T}_{\ell,u}=\frac{\ell^{2}}{\ell^{2}+1}\mathcal{M}\text{.}%
\]
Therefore, there is a constant $c>0$ (depending only on the masses) such that
any eigenvalue $\lambda$ of $\hat{T}_{\ell}$ satisfies $\left\vert
\lambda\right\vert \geq c^{-1}$. We conclude that the matrix $\hat{T}_{\ell}$
in \eqref{G_0 blocks} is invertible and we write%
\[
\partial_{\eta}F_{0}[(x_{a},0)]^{-1}\eta=\sum_{\ell\neq0}\hat{T}_{\ell}%
^{-1}\hat{x}_{\ell}e_{\ell},\quad\eta\in W^{\Gamma}.
\]
It follows that
\begin{equation}
\Vert\partial_{\eta}F_{0}[(x_{a},0)]^{-1}\eta\Vert\leq c\Vert\eta
\Vert.\nonumber
\end{equation}
Note that the Hessian $\nabla^{2}\mathcal{A}_{0}[gx_{a}]$ is conjugated to
$\nabla^{2}\mathcal{A}_{0}[x_{a}]$ because $\nabla\mathcal{A}_{0}$ is
$G$-equivariant. Hence $\partial_{\eta}F_{0}[(gx_{a},0)]$ and $\partial_{\eta
}F_{0}[(x_{a},0)]$ are conjugated. Therefore, the estimate for $\partial
_{\eta}F_{0}[(gx_{a},0)]$ holds independently of $g$ because the group $G$
acts by isometries.
\end{proof}

\begin{remark}
In the plane ($d=1$) and for the Newton gravitational force ($\alpha=2$), the
operator $\partial_{\eta}F_{0}[(gx_{a},0)]\ $is not invertible because
$\sqrt{3-\alpha}=1$ and $\lambda_{1,1}^{-}=0$, which is a consequence of the
fact that circular orbits of the Kepler problem with gravitational potential
are never isolated due to the existence of elliptic orbits. In the case of
more dimensions ($d>1$), the operator $\partial_{\eta}F_{0}[(gx_{a},0)]\ $is
never invertible in $W$ due to resonances of the circular orbit of the
generalized Kepler problem with its rotations in more dimensions. In both
cases, the operators are invertible only when we restrict the operator to
$W^{\Gamma}$.
\end{remark}

\begin{theorem}
[Lyapunov-Schmidt reduction]\label{thm:Lyapunov-Schmidt reduction} Assume
conditions $\mathbf{(A)}-\mathbf{(B)}$. Under one of the assumptions
$\mathbf{(C1)}-\mathbf{(C3)}$, there is $\varepsilon_{0}>0$ such that, for
every $\varepsilon\in(0,\varepsilon_{0})$, there is an open neighbourhood
$\mathcal{V}\subset X_{0}^{\Gamma}$ of the orbit $G(x_{a})$ and a smooth
$H$-equivariant mapping $\varphi_{\varepsilon}:\mathcal{V}\rightarrow
W^{\Gamma}$ such that solving $\nabla\mathcal{A}(\xi+\eta)=0$ for $\xi
\in\mathcal{V}$ is equivalent to solve the finite dimensional system of
equations $\nabla\Psi_{\varepsilon}(\xi)=0$ for $\xi\in\mathcal{V}$, where
\[
\Psi_{\varepsilon}(\xi)=\mathcal{A}(\xi+\varphi_{\varepsilon}(\xi))
\]
is the reduced functional. The fact that, for $\varepsilon\in(0,\varepsilon
_{0})$, the operator $F_{\varepsilon}(\xi,\eta)$ is analytic implies that the
implicit function $\varphi_{\varepsilon}(\xi)$ is also analytic.
\end{theorem}

\begin{proof}
Lemma \ref{Estimate hessian A_0} ensures that, for every $g\in G$, the
operator $\partial_{\eta}F_{0}[(gx_{a},0)]$ restricted to $W^{\Gamma}$ has
bounded inverse. The implicit function theorem assures the existence of open
neighbourhoods $\mathcal{I}^{g}\subset\mathbb{R}$ of $0$ and $\mathcal{V}%
^{g}\subset X_{0}^{\Gamma}$ of $gx_{a}$ such that, for every $\varepsilon
\in\mathcal{I}^{g}$, there is a unique smooth mapping $\varphi_{\varepsilon
}^{g}:\mathcal{V}^{g}\rightarrow W^{\Gamma}$ such that the solutions of
\[
F_{\varepsilon}(\xi,\varphi_{\varepsilon}^{g}(\xi))=0\text{,}\quad\xi
\in\mathcal{V}^{g}%
\]
lie on $\eta=\varphi_{\varepsilon}^{g}(\xi)$. Since this argument is valid for
every $g\in G$, we can repeat this procedure until we obtain a cover of the
orbit $G(x_{a})$ by open sets $\mathcal{V}^{g}\subset X_{0}^{\Gamma}$ from
which we can extract a finite cover $\{\mathcal{V}^{g_{i}}\}_{i=1}^{n}$, by
compactness of the group orbit. We define open sets $\mathcal{V}=\cup
_{i=1}^{n}\mathcal{V}^{g_{i}}$ and $\mathcal{I}=\cap_{i=1}^{n}\mathcal{I}%
^{g_{i}}$. We take $\varepsilon_{0}$ small enough such that $(0,\varepsilon
_{0})\subset\mathcal{I}$. Hence for $\varepsilon\in(0,\varepsilon_{0})$ there
is a smooth mapping $\varphi_{\varepsilon}:\mathcal{V}\rightarrow W^{\Gamma}$,
defined by $\varphi_{\varepsilon}(\xi)=\varphi_{\varepsilon}^{g_{i}}(\xi)$
whenever $\xi\in\mathcal{V}^{g_{i}}$, such that the solutions of
\begin{equation}
F_{\varepsilon}(\xi,\varphi_{\varepsilon}(\xi))=0\text{,}\quad\xi
\in\mathcal{V} \label{pf 3.2: F_e equation}%
\end{equation}
lie on $\eta=\varphi_{\varepsilon}(\xi)$. Since $F_{\varepsilon}$ is
$H$-equivariant, both functions $(\xi,g^{-1}\varphi_{\varepsilon}(g\xi))$ and
$(\xi,\varphi_{\varepsilon}(\xi))$ are solutions of
\eqref{pf 3.2: F_e equation} for any $g\in H$. By uniqueness of solutions,
$\varphi_{\varepsilon}$ is $H$-equivariant. Note that we may have to take
$\mathcal{V}$ smaller such that if $\xi\in\mathcal{V}$ then $\xi
+\varphi_{\varepsilon}(\xi)$ lies in the open set $\Omega^{\Gamma}$, which is
the open neighbourhood in $X^{\Gamma}$ of $G(x_{a})$ we started with.

For fixed $\varepsilon\in(0,\varepsilon_{0})$ define the reduced functional
$\Psi_{\varepsilon}:\mathcal{V}\subset X_{0}^{\Gamma}\rightarrow\mathbb{R}$ by $\Psi_{\varepsilon}
(\xi):=\mathcal{A}(\xi+\varphi_{\varepsilon}(\xi))$. Then
\[
\nabla\Psi_{\varepsilon}(\xi)=P\nabla\mathcal{A}(\xi+\varphi_{\varepsilon}(\xi
))+\nabla_{\eta}\mathcal{A}(\xi+\varphi_{\varepsilon}(\xi))D_{\xi}%
\varphi_{\varepsilon}(\xi)=P\nabla\mathcal{A}(\xi+\varphi_{\varepsilon}%
(\xi)).
\]
Hence $\nabla\mathcal{A}(\xi+\eta)=0$ with $\xi\in\mathcal{V}$ if and only if
$\eta=\varphi_{\varepsilon}(\xi)$ and $\nabla\Psi_{\varepsilon}(\xi)=0$.
\end{proof}

\subsection{Estimate for the reduced functional}

Fix $\varepsilon\in(0,\varepsilon_{0})$ and write the reduced functional
$\Psi_{\varepsilon}:\mathcal{V}\rightarrow\mathbb{R}$ as $\Psi_{\varepsilon}(\xi)=\mathcal{A}_{0}%
(\xi)+\mathcal{N}(\xi)$, where
\[
\mathcal{N}(\xi)=\mathcal{A}_{0}(\xi+\varphi_{\varepsilon}(\xi))-\mathcal{A}%
_{0}(\xi)+\mathcal{H}(\xi+\varphi_{\varepsilon}(\xi)).
\]
The terms $\mathcal{A}_{0}(\xi)$ and $\mathcal{A}_{0}(\xi+\varphi
_{\varepsilon}(\xi))$ blow up as $\varepsilon\rightarrow0$ for $\alpha>1$. The
core of the main theorem resides in obtaining uniform estimates for
$\nabla\mathcal{N}(\xi)$. While the matrix $\mathcal{D}_{\varepsilon}$ scales
correctly the equation $\nabla_{\eta}\mathcal{A}(\xi+\eta)=0$, we need to
define another matrix that scales correctly the equation $\nabla_{\xi
}\mathcal{A}(\xi+\eta)=0$. Let
\begin{equation}
\mathcal{C}_{\varepsilon}:=\varepsilon^{\alpha-1}\mathcal{I}\oplus
\mathcal{I}\oplus\dots\oplus\mathcal{I}.
\end{equation}

\begin{lemma}
\label{lemma: estimate phi} Assume conditions $\mathbf{(A)}-\mathbf{(B)}$.
Under one of the assumptions $\mathbf{(C1)}-\mathbf{(C3)}$, there is a
constant $N_{1}>0$, independent of the parameter $\varepsilon\in
(0,\varepsilon_{0})$, such that
\[
\Vert\varphi_{\varepsilon}(\xi)\Vert\leq N_{1}(\varepsilon+\Vert\xi
-gx_{a}\Vert^{2})\quad\text{for every}\quad\xi\in\mathcal{V},~g\in G,
\]
where we may have to take a smaller neighborhood $\mathcal{V}$ of $G(x_{a})$.
\end{lemma}

\begin{proof}
By theorem \ref{thm:Lyapunov-Schmidt reduction} the implicit mapping
$\varphi_{\varepsilon}(\xi)$ solves the equation
\[
\nabla_{\eta}\mathcal{A}_{0}(\xi+\varphi_{\varepsilon}(\xi))=-\nabla_{\eta
}\mathcal{H}(\xi+\varphi_{\varepsilon}(\xi)).
\]
for $\xi\in\mathcal{V}$. Since we can take $\varepsilon_{0}<1$ and $(I-P)$ is
a projection, there is a constant $N_{2}>0$ such that
\begin{equation}
\Vert\nabla_{\eta}\mathcal{A}_{0}(\xi+\varphi_{\varepsilon}(\xi))\Vert
=\Vert\nabla\mathcal{H}(\xi+\varphi_{\varepsilon}(\xi))\Vert\leq
N_{2}\varepsilon\label{AH}%
\end{equation}
by Lemma \ref{lemma: bound for gradient of H}. Define the operator
$\mathcal{L}:X\rightarrow X$ by
\[
\mathcal{L}=M_{0}\left(  \frac{\nu}{\omega}\partial_{s}+\mathcal{J}\right)
^{2}\oplus M_{1}\left(  \nu\partial_{s}+\mathcal{J}\right)  ^{2}%
\oplus...\oplus M_{n}\left(  \nu\partial_{s}+\mathcal{J}\right)  ^{2}.
\]
For $x\in X$ given by $x(s)=(u_{0}(s),u(s))$, set
\[
U_{0}(x)=M_{0}\phi_{\alpha}(\Vert u_{0}\Vert)+\sum_{1\leq j<k\leq n}M_{j}%
M_{k}\phi_{\alpha}(\Vert u_{j}-u_{k}\Vert).
\]
We have that
\begin{equation}
\mathcal{C}_{\varepsilon}\nabla_{\eta}\mathcal{A}_{0}(x)=(I-P)\left(
-\partial_{s}^{2}+1\right)  ^{-1}(-\mathcal{L}x+\nabla U_{0}(x)).
\label{main:eq3bis}%
\end{equation}

Hereafter we use the fact that the differential operator $\left(
-\partial_{s}^{2}+1\right)  ^{-1}\mathcal{L}:X\rightarrow X$ and the
projection $(I-P)$ commute, because they are block diagonal operators in
Fourier components (\ref{OpFou}). Thus $(I-P)\left(  -\partial_{s}%
^{2}+1\right)  ^{-1}\mathcal{L}\xi=0$ for any $\xi\in\mathcal{V}$ and
\[
\mathcal{C}_{\varepsilon}\nabla_{\eta}\mathcal{A}_{0}(\xi+\varphi
_{\varepsilon}(\xi))=(I-P)\left(  -\partial_{s}^{2}+1\right)  ^{-1}\left(
-\mathcal{L}\varphi_{\varepsilon}(\xi)+\nabla U_{0}(\xi+\varphi_{\varepsilon
}(\xi))\right)  ,
\]
for any $\xi\in\mathcal{V}$.

Since $X$ is a Banach algebra and $U_{0}(x)$ is analytic in $\Omega\subset X$,
we can perform a Taylor expansion of $\nabla U_{0}(\xi+\varphi_{\varepsilon
}(\xi))$ around $\xi=x_{a}$ in $X$. In particular, there is a ball
$\mathcal{B}_{\delta}\subset\mathcal{V}$ of radius $\delta>0$ (independent of
the parameter $\varepsilon$ because $U_{0}$ does not depend on $\varepsilon$)
centered at $x_{a}$ such that, if $\xi\in\mathcal{B}_{\delta}$, the following
inequality holds
\[
\left\Vert \nabla U_{0}(\xi+\varphi_{\varepsilon}(\xi))-\nabla^{2}U_{0}%
[x_{a}]\left(  \xi-x_{a}+\varphi_{\varepsilon}(\xi)\right)  \right\Vert \leq
N_{3}\Vert\xi-x_{a}+\varphi_{\varepsilon}(\xi)\Vert^{2}
\]
for some positive constant $N_{3}$. Since the norms of the operator $\left(
-\partial_{s}^{2}+1\right)  ^{-1}:X\rightarrow X$ and $(I-P):X\rightarrow W$
are smaller or equal to $1$ then, for $\xi\in\mathcal{B}_{\delta}$,
\begin{equation}
\Vert\mathcal{C}_{\varepsilon}\nabla_{\eta}\mathcal{A}_{0}(\xi+\varphi
_{\varepsilon}(\xi))-\mathcal{C}_{\varepsilon}\nabla_{\eta}^{2}\mathcal{A}%
_{0}[x_{a}]\varphi_{\varepsilon}(\xi)\Vert\leq N_{3}\Vert\xi-x_{a}%
+\varphi_{\varepsilon}(\xi)\Vert^{2}. \label{eq: main 3}%
\end{equation}
By the triangle inequality,
\[
\Vert\mathcal{C}_{\varepsilon}\nabla_{\eta}^{2}\mathcal{A}_{0}[x_{a}%
]\varphi_{\varepsilon}(\xi)\Vert\leq\Vert\mathcal{C}_{\varepsilon}\nabla
_{\eta}\mathcal{A}_{0}[\xi+\varphi_{\varepsilon}(\xi)]\Vert+N_{3}\Vert
\xi-x_{a}+\varphi_{\varepsilon}(\xi)\Vert^{2}.
\]

Since $\left\Vert \mathcal{D}_{\varepsilon}\right\Vert \leq\left\Vert
\mathcal{C}_{\varepsilon}\right\Vert \leq1$ if $\varepsilon_{0}<1$, we
conclude using (\ref{AH}) that
\begin{equation}
\Vert\mathcal{D}_{\varepsilon}\nabla_{\eta}^{2}\mathcal{A}_{0}[x_{a}%
]\varphi_{\varepsilon}(\xi)\Vert\leq N_{2}\varepsilon+N_{3}\Vert\xi-x_{a}%
\Vert^{2}+N_{3}\Vert\varphi_{\varepsilon}(\xi)\Vert^{2}. \label{eq:main 5}%
\end{equation}
In lemma \ref{Estimate hessian A_0} we obtained a uniform bound $c>0$ for the
inverse of the operator $\partial_{\eta}F_{0}[x_{a}]=\lim_{\varepsilon
\rightarrow0}\mathcal{D}_{\varepsilon}\nabla_{\eta}^{2}\mathcal{A}_{0}[x_{a}%
]$. Since $\mathcal{D}_{\varepsilon}\nabla_{\eta}^{2}\mathcal{A}_{0}[x_{a}]$
is continuous at $\varepsilon=0$, then
\[
\Vert\left(  \mathcal{D}_{\varepsilon}\nabla_{\eta}^{2}\mathcal{A}_{0}%
[x_{a}]\right)  ^{-1}\Vert\leq2\Vert\partial_{\eta}F_{0}[x_{a}]^{-1}\Vert
\leq2c
\]
for $\varepsilon\in(0,\varepsilon_{0})$ with $\varepsilon_{0}$ small enough.
Taking $\eta=\mathcal{D}_{\varepsilon}\nabla_{\eta}^{2}\mathcal{A}_{0}%
[x_{a}]\varphi_{\varepsilon}(\xi)$, we conclude that
\[
\Vert\varphi_{\varepsilon}(\xi)\Vert=\left\Vert \left(  \mathcal{D}%
_{\varepsilon}\nabla_{\eta}^{2}\mathcal{A}_{0}[x_{a}]\right)  ^{-1}%
\eta\right\Vert \leq2c\Vert\eta\Vert=2c\Vert\mathcal{D}_{\varepsilon}%
\nabla_{\eta}^{2}\mathcal{A}_{0}[x_{a}]\varphi_{\varepsilon}(\xi)\Vert~.
\]
By \eqref{eq:main 5} and the previous inequality we obtain%
\[
\frac{1}{2c}\Vert\varphi_{\varepsilon}(\xi)\Vert\leq N_{2}\varepsilon
+N_{3}\Vert\xi-x_{a}\Vert^{2}+N_{3}\Vert\varphi_{\varepsilon}(\xi)\Vert^{2}.
\]
By choosing the ball radius $\delta$ small enough such that $N_{3}\Vert
\varphi_{\varepsilon}(\xi)\Vert<\frac{1}{4c}$ we get
\[
\Vert\varphi_{\varepsilon}(\xi)\Vert\leq4c\left(  N_{2}\varepsilon+N_{3}%
\Vert\xi-x_{a}\Vert^{2}\right)  ,
\]
whenever $\xi\in\mathcal{B}_{\delta}$. We obtain the result with
$N_{1}:=4c\max\left(  N_{2},N_{3}\right)  .$

This procedure gives the constant $N_{1}$ of the statement independent of
$\varepsilon\in(0,\varepsilon_{0})$. This estimate holds on a neighbourhood of
the orbit $G(x_{a})$ and not only in a neighbourhood of $x_{a}$. Indeed, since
the constants $N_{2}$ and $c$ do not depend on the point of the orbit, we
could work around another point $gx_{a}$ of the orbit and obtain the same
estimates in a ball $\mathcal{B}_{\delta^{g}}\subset\mathcal{V}$. By
compactness of the orbit, there is $\delta>0$ such that the orbit can be
covered by balls of radius $\delta$ and the estimate \eqref{eq:
main 3} holds at each point of the orbit. Therefore, all the estimates are
valid in the union of balls of radius $\delta$ that we rename $\mathcal{V}$.
\end{proof}

\begin{theorem}
[Uniform estimate]\label{N} Assume conditions $\mathbf{(A)}-\mathbf{(B)}$.
Under one of the assumptions $\mathbf{(C1)}-\mathbf{(C3)}$, the reduced
functional $\Psi_{\varepsilon}:\mathcal{V}\rightarrow\mathbb{R}$ can be written as $\Psi_{\varepsilon}
(\xi)=\mathcal{A}_{0}(\xi)+\mathcal{N}(\xi)\text{,}$ where $\mathcal{N}(\xi)$
is $H$-equivariant and satisfies the uniform estimate
\[
\Vert\mathcal{C}_{\varepsilon}\nabla\mathcal{N}(\xi)\Vert\leq N(\varepsilon
+\Vert\xi-gx_{a}\Vert^{2}),
\]
for all $\varepsilon\in(0,\varepsilon_{0})$ and $g\in G$, with $N>0$ a
constant independent on the parameters.
\end{theorem}

\begin{proof}
Note that
\[
\mathcal{C}_{\varepsilon}\nabla\mathcal{N}(\xi)=\mathcal{C}_{\varepsilon
}P\left[  \nabla\mathcal{A}_{0}(\xi+\varphi_{\varepsilon}(\xi))-\nabla
\mathcal{A}_{0}(\xi)\right]  +\mathcal{C}_{\varepsilon}P\nabla\mathcal{H}%
(\xi+\varphi_{\varepsilon}(\xi)).
\]
Since the operator norms of $\mathcal{C}_{\varepsilon}$ and $P$ are bounded by
$1$, there is a constant $N_{2}>0$ such that $\left\Vert P\mathcal{C}%
_{\varepsilon}\nabla\mathcal{H}(\xi+\varphi_{\varepsilon}(\xi))\right\Vert
\leq N_{2}\varepsilon$. By the triangle inequality%
\[
\Vert\mathcal{C}_{\varepsilon}\nabla\mathcal{N}(\xi)\Vert\leq\Vert
\mathcal{C}_{\varepsilon}P\left[  \nabla\mathcal{A}_{0}(\xi+\varphi
_{\varepsilon}(\xi))-\nabla\mathcal{A}_{0}(\xi)\right]  \Vert+N_{2}%
\varepsilon\text{.}%
\]
Applying the mean value theorem, there is some $\mu\in\lbrack0,1]$ such that
\begin{equation}
\mathcal{C}_{\varepsilon}\left[  \nabla\mathcal{A}_{0}(\xi+\varphi
_{\varepsilon}(\xi))-\nabla\mathcal{A}_{0}(\xi)\right]  =\mathcal{C}%
_{\varepsilon}\nabla^{2}\mathcal{A}_{0}[\xi+\mu\varphi_{\varepsilon}%
(\xi)]\varphi_{\varepsilon}(\xi).\label{mean value thm}%
\end{equation}
Using the notations of the previous lemma, the Hessian reads
\[
\mathcal{C}_{\varepsilon}\nabla^{2}\mathcal{A}_{0}[\xi+\mu\varphi
_{\varepsilon}(\xi)]=\left(  -\partial_{s}^{2}+1\right)  ^{-1}\left(
-\mathcal{L}+\nabla^{2}U_{0}[\xi+\mu\varphi_{\varepsilon}(\xi)]\right)
\text{.}%
\]
Since the operator $\mathcal{L}$ commutes with $P$ and $P\varphi_{\varepsilon
}(\xi)=0$, then
\[
P\mathcal{C}_{\varepsilon}\nabla^{2}\mathcal{A}_{0}[\xi+\mu\varphi
_{\varepsilon}(\xi)]\varphi_{\varepsilon}(\xi)=P\left(  -\partial_{s}%
^{2}+1\right)  ^{-1}\nabla^{2}U_{0}[\xi+\mu\varphi_{\varepsilon}(\xi
)]\varphi_{\varepsilon}(\xi).
\]
Therefore by \eqref{mean value thm} and the fact that the norm of $\left(
-\partial_{s}^{2}+1\right)  ^{-1}$ is bounded by $1$, we obtain%
\[
\Vert P\mathcal{C}_{\varepsilon}\left(  \nabla\mathcal{A}_{0}(\xi
+\varphi_{\varepsilon}(\xi))-\nabla\mathcal{A}_{0}(\xi)\right)  \Vert\leq\Vert
P\nabla^{2}U_{0}[\xi+\mu\varphi_{\varepsilon}(\xi)]\varphi_{\varepsilon}%
(\xi)\Vert\leq e\Vert\varphi_{\varepsilon}(\xi)\Vert,
\]
for some constant $e>0$ independent of $\mu$, which exists because the
operator $P\nabla^{2}U_{0}[\xi+\mu\varphi_{\varepsilon}(\xi)]$ is bounded
independently of the parameter $\varepsilon$ because $U_{0}$ does not depend
on $\varepsilon$. The result of the statement follows from lemma
\ref{lemma: estimate phi} by setting $N:=eN_{1}+N_{2}$.
\end{proof}

\section{Critical points of the reduced functional}

Let us summarise what we achieved so far. Suppose $\varepsilon\in
(0,\varepsilon_{0})$ and conditions $\mathbf{(A)-(B)}$ are satisfied. Then,
under one of the assumptions $\mathbf{(C1)}-\mathbf{(C3)}$, there is a
neighbourhood $\mathcal{V}\subset X_{0}^{\Gamma}$ of the orbit $G(x_{a})$ such
that the problem of finding a solution $x=\xi+\eta\in X^{\Gamma}$ of the
Euler-Lagrange equations \eqref{gradient A_0 u} is reduced to finding a
solution $\xi\in\mathcal{V}\subset X_{0}^{\Gamma}$ of $\nabla\Psi
_{\varepsilon}(\xi)=0$. Furthermore, the reduced functional is given by
\[
\Psi_{\varepsilon}(\xi)=\mathcal{A}_{0}(\xi)+\mathcal{N}(\xi),
\]
where $\mathcal{A}_{0}(\xi)$ is $G$-invariant, $\mathcal{H}(\xi)$ is
$H$-invariant, and $\varphi_{\varepsilon}(\xi)$ is $H$-equivariant, where
$H\subset G$.

The critical points of $\Psi_{\varepsilon}(\xi)$ cannot be obtained directly
by a continuation of solutions of $\nabla\Psi_{\varepsilon}(\xi)=0$ using the
parameter $\varepsilon\in(0,\varepsilon_{0})$ because $\varepsilon$ encodes
the distance between the pair of bodies, and the function $\nabla
\Psi_{\varepsilon}(\xi)$ explodes as $\varepsilon\rightarrow0$ when $\alpha
>1$. Before proceeding with the continuation of solutions we need to solve
first the singular part of $\nabla\Psi_{\varepsilon}(\xi)$. For logarithm
potentials (case $\alpha=1$), it is still possible to continue the solutions
directly from $\Psi_{\varepsilon}(\xi)$ for $\varepsilon=0$. For instance, in
\cite{Ba17}, this approach is used for a Hamiltonian system corresponding to
the $n$-vortex problem.

\subsection{The regular functional}

In this section we obtain a regular functional by passing to the quotient
space $\xi\in\mathcal{V}\subset X_{0}^{\Gamma}$ under the action of the group
$H$. Let
\[
\xi=(\xi_{0},\xi_{1})\in X_{0}^{\Gamma}=E_{0}\times E^{\prime}.
\]
and recall that the group $G=G_{1}\times G_{2}$ acts diagonally on
$X_{0}^{\Gamma}$. Under the conditions $\mathbf{(C1)-(C3)}$, we have that
$G_{1}\left(  a_{0}\right)  $ is a unit circle $S(E_{0})$ in $E_{0}$. In the case $\mathbf{(C1)-(C2)}$ this
follows from the fact that $E_{0}=E$ is the plane and $G_{1}$ acts as
$U(1)\ $on the plane. In the case $\mathbf{(C3)}$ we have that $E_{0}$ is the plane $\Pi\subset E$ and
$G_{1}=U(1)\times U(d-1)$ acts as $U(1)\ $on the plane $E_{0}=\Pi$.

Notice that we chose $a_{0}=(1,0,\dots,0)$ for $d\geq1$. Thus for every
$\xi_{0}\in E_{0}$ we can find $h\in G_{1}$ such that $\xi_{0}=rha_{0}$ for
some $r\in\mathbb{R}$. Since $H=\widetilde{G_{1}}$, we obtain
\[
\Psi_{\varepsilon}(\xi_{0},\xi_{1})=\Psi_{\varepsilon}(h^{-1}\xi_{0},h^{-1}%
\xi_{1})=\Psi_{\varepsilon}(ra_{0},h^{-1}\xi_{1})\text{,}%
\]
by using $H$-invariance. Setting $\xi^{\prime}=h^{-1}\xi_{1}$ one obtains that
$\Psi_{\varepsilon}(\xi_{0},\xi_{1})=\Psi_{\varepsilon}(ra_{0},\xi^{\prime})$
depends only on the variables $\left(  r,\xi^{\prime}\right)  $. In particular
the solutions of $\nabla\Psi_{\varepsilon}(\xi_{0},\xi_{1})=0$ are in one to
one correspondence with the solutions of
\[
\partial_{r}\Psi_{\varepsilon}(ra_{0},\xi^{\prime})=0\quad\mbox{and}\quad
\nabla_{\xi^{\prime}}\Psi_{\varepsilon}(ra_{0},\xi^{\prime})=0.
\]
Furthermore, we observe that the function $\Psi_{\varepsilon}(ra_{0}%
,\xi^{\prime})$ is $H_{a_{0}}$-invariant, where
\[
H_{a_{0}}:=\left\{  g\in G_{2}\mid\;g\in(G_{1})_{a_{0}}\right\}
\]
is the stabiliser of $a_{0}$ in $H$. Note that $H_{a_{0}}$ is only acting on
the second component because it is a subgroup of $G_{2}$.

\begin{remark}
In the case $d=1$, we can use polar coordinates to write $\xi_{0}=re^{i\theta
}$. Similarly $\xi_{1}=(\rho_{1}e^{i\theta_{1}%
},...,\rho_{n}e^{i\theta_{n}})$. Then the 'reduced' variable is $\xi^{\prime
}=(\rho_{1}e^{i\theta_{1}^{\prime}},...,\rho_{n}e^{i\theta_{n}^{\prime}})$
where ${\theta}_{j}^{\prime}=\theta_{j}-\theta$.
\end{remark}

\begin{theorem}
\label{main regular} Under conditions $\mathbf{(C1)-(C3)}$, for $\varepsilon
\in(0,\varepsilon_{0})$, the critical points of $\Psi_{\varepsilon}(\xi)$ in
the (possibly smaller) neighbourhood $\mathcal{V}\subset X_{0}^{\Gamma}$ are
in one to one correspondence with the critical points of the $H_{a_{0}}%
$-invariant function $\Psi_{\varepsilon}^{\prime}:\mathcal{V}^{\prime}\subset
E^{\prime}\rightarrow\mathbb{R}$ given by
\[
\Psi_{\varepsilon}^{\prime}(\xi^{\prime})=V(\xi^{\prime})+\mathcal{N}^{\prime
}(\xi^{\prime})\text{,}%
\]
where $\mathcal{V}^{\prime}\subset E^{\prime}$ is a neighbourhood of
$G_{2}(a)$, $V(\xi^{\prime})$ is the amended potential as in
\eqref{amended potential V} and
\[
\mathcal{N}^{\prime}(\xi^{\prime})=\mathcal{A}_{0}(r_{\varepsilon}\left(
\xi^{\prime}\right)  a_{0},\xi^{\prime})-V(\xi^{\prime})+\mathcal{N}%
(r_{\varepsilon}(\xi^{\prime})a_{0},\xi^{\prime})
\]
where $r_{\varepsilon}:\mathcal{V}^{\prime}\subset E^{\prime}\rightarrow
\mathbb{R}$ is the unique $H_{a_{0}}$-invariant function that solves the
equation $\partial_{r}\Psi_{\varepsilon}(r_{\varepsilon}(\xi^{\prime}%
)a_{0},\xi^{\prime})=0$. Furthermore there are constants $N^{\prime}%
,N_{1}^{\prime}>0$ such that for each $g\in G_{2}$,
\[
\Vert r_{\varepsilon}(\xi^{\prime})\Vert\leq N_{1}^{\prime}(\varepsilon
+\Vert\xi^{\prime}-ga\Vert^{2})\quad\mbox{and}\quad\Vert\nabla_{\xi^{\prime}%
}\mathcal{N}^{\prime}(\xi^{\prime})\Vert\leq N^{\prime}(\varepsilon+\Vert
\xi^{\prime}-ga\Vert^{2}).
\]

\end{theorem}

\begin{proof}
The function $\Psi_{\varepsilon}(ra_{0},\xi^{\prime})$ reads
\[
\Psi_{\varepsilon}(ra_{0},\xi^{\prime})=\mathcal{A}_{0}(ra_{0},\xi^{\prime
})+\mathcal{N}(ra_{0},\xi^{\prime})
\]
where $(ra_{0},\xi^{\prime})\in\mathcal{V}\subset X_{0}^{\Gamma}$ and
\begin{equation}
\mathcal{A}_{0}(ra_{0},\xi^{\prime})=2\pi\left(  \varepsilon^{1-\alpha}%
M_{0}\left(  \frac{1}{2}r^{2}+\phi_{\alpha}(r)\right)  +V(\xi^{\prime
})\right)  .\nonumber
\end{equation}
We want to express $r$ as a function of $\xi^{\prime}$ from the equation
$\partial_{r}\Psi_{\varepsilon}(ra_{0},\xi^{\prime})=0$. Using the same
strategy as before, we consider the regularised $r$-gradient
\[
f_{\varepsilon}(r,\xi^{\prime})=\varepsilon^{\alpha-1}\partial_{r}%
\Psi_{\varepsilon}(ra_{0},\xi^{\prime}).
\]
Observe that
\[
f_{\varepsilon}(r,\xi^{\prime}):=2\pi M_{0}\left(  r-\frac{1}{r^{\alpha}%
}\right)  +\varepsilon^{\alpha-1}\partial_{r}\mathcal{N}(ra_{0},\xi^{\prime}).
\]
By Theorem \ref{N} the regularised $r$-gradient extends continuously at
$\varepsilon=0$ and $f_{0}(1,a)=0$. Thus, in order to apply the implicit
function theorem we only need to show that the derivative
\begin{equation}
\partial_{r}f_{0}\left(  1,a\right)  =2\pi M_{0}(\alpha+1)+\lim_{\varepsilon
\rightarrow0}\varepsilon^{\alpha-1}\partial_{r}^{2}\mathcal{N}\left(
x_{a}\right)  \label{N vanishes}%
\end{equation}
is non zero. Since $ra_{0}=(r,0,...,0)$, then $\varepsilon^{\alpha-1}%
\partial_{r}^{2}\mathcal{N}\left(  x_{a}\right)  $ is the first row of
$\mathcal{C}_{\varepsilon}\nabla^{2}\mathcal{N}\left(  x_{a}\right)  $ and
\[
\left\Vert \varepsilon^{\alpha-1}\partial_{r}^{2}\mathcal{N}\left(
x_{a}\right)  \right\Vert \leq\left\Vert \mathcal{C}_{\varepsilon}\nabla
^{2}\mathcal{N}\left(  x_{a}\right)  \right\Vert .
\]
The fact that $\mathcal{C}_{\varepsilon}\nabla\mathcal{N}(\xi)$ is analytic in
$\varepsilon$ and the uniform estimate in Theorem \ref{N} imply that
$\Vert\mathcal{C}_{\varepsilon}\nabla\mathcal{N}\left(  \xi\right)
|_{\varepsilon=0}\Vert\leq N\Vert\xi-gx_{a}\Vert^{2}$. This inequality implies
that $\nabla_{\xi}\mathcal{N}(\xi)$ has no linear term at $\xi=x_{a}$, i.e.
its linearisation is zero $\lim_{\varepsilon\rightarrow0}\mathcal{C}%
_{\varepsilon}\nabla^{2}\mathcal{N}\left(  x_{a}\right)  =0$. Thus
\eqref{N vanishes} is non-vanishing.

By the implicit function theorem we conclude that there is a smooth function
$r_{\varepsilon}$ defined on a neighbourhood $\mathcal{V}^{\prime}\subset
E^{\prime}$ of $a$ such that
\[
f_{\varepsilon}(r_{\varepsilon}(\xi^{\prime}),\xi^{\prime})=\varepsilon
^{\alpha-1}\partial_{r}\Psi_{\varepsilon}(r_{\varepsilon}(\xi^{\prime}%
),\xi^{\prime})=0
\]
on this neighbourhood. As before this argument can be repeated at any point of
the orbit $G_{2}(a)$ in $E^{\prime}$ and we can assume that $\mathcal{V}%
^{\prime}$ is a neighbourhood of $G_{2}(a)$ in $E^{\prime}$. Hence, when we
fix $\varepsilon\in(0,\varepsilon_{0})$ and take a smaller neighbourhood
$\mathcal{V}\subset X_{0}^{\Gamma}$, the critical points of $\Psi
_{\varepsilon}(ra_{0},\xi^{\prime})$ in $\mathcal{V}$ are in one to one
correspondence with the critical points of the function $\Psi_{\varepsilon
}^{\prime}:\mathcal{V}^{\prime}\subset E^{\prime}\rightarrow\mathbb{R}$ given
by
\[
\Psi_{\varepsilon}^{\prime}(\xi^{\prime})=\mathcal{A}_{0}(r_{\varepsilon
}\left(  \xi^{\prime}\right)  a_{0},\xi^{\prime})+\mathcal{N}(r_{\varepsilon
}(\xi^{\prime})a_{0},\xi^{\prime})=V(\xi^{\prime})+\mathcal{N}^{\prime}%
(\xi^{\prime}).
\]
By uniqueness of $r_{\varepsilon}$ and $H_{a_{0}}$-equivariance of
$\Psi_{\varepsilon}(ra_{0},\xi^{\prime})$, we have that $r_{\varepsilon}%
$ is $H_{a_{0}}$-invariant, i.e. $\Psi_{\varepsilon}^{\prime
}$ is $H_{a_{0}}$-invariant. By Theorem \ref{N}, we have for
$g\in G_{2}$ the uniform estimates
\begin{align*}
|\varepsilon^{\alpha-1}\partial_{r}\mathcal{N}(ra_{0},\xi^{\prime})|  &  \leq
N(\varepsilon+\Vert\xi^{\prime}-ga\Vert^{2}+|r-1|^{2}),\\
\Vert\nabla_{\xi^{\prime}}\mathcal{N}(ra_{0},\xi^{\prime})\Vert &  \leq
N(\varepsilon+\Vert\xi^{\prime}-ga\Vert^{2}+|r-1|^{2}).
\end{align*}
Using these estimates and an argument analogous to Lemma
\ref{lemma: estimate phi}, it is possible to obtain the uniform estimates for
$r_{\varepsilon}(\xi^{\prime})$ and $\mathcal{N}^{\prime}(\xi^{\prime})$.
\end{proof}

\subsection{Critical points of the regular functional}

In this section we find the critical points of the regular functional
$\Psi_{\varepsilon}^{\prime}(\xi^{\prime})=V(\xi^{\prime})+\mathcal{N}%
^{\prime}(\xi^{\prime})$, where $V(\xi^{\prime})$ is $G_{2}$-invariant and
$\mathcal{N}^{\prime}(\xi^{\prime})$ is $H_{a_{0}}$-invariant. The potential
$\Psi_{0}^{\prime}(\xi^{\prime})=V(\xi^{\prime})$ has the orbit of critical
points $G_{2}\left(  a\right)  $. Thus, we encounter a similar situation to
the case studied in \cite{Fo} where the term $\mathcal{N}^{\prime}$ breaks the
symmetry from $G_{2}$ to the subgroup $H_{a_{0}}$.

Next we use Palais slice coordinates for $\xi^{\prime}$. Let $K:=\left(
G_{2}\right)  _{a}$ be the stabiliser of $a$ and $G_{2}(a)\subset E^{\prime}$
be the group orbit of $a$. Let $W=E^{\prime}/T_{a}G_{2}(a)$ be a $K$-invariant
complement in $E^{\prime}$. By the Palais slice theorem, there is a
$K$-invariant neighbourhood of $0$ denoted $W_{0}\subset W$, and a $G_{2}%
$-invariant neighbourhood of $G_{2}(a)$ denoted $\mathcal{V}^{\prime}\subset
E^{\prime}$, such that $\mathcal{V}^{\prime}$ is isomorphic to the associated
bundle $G_{2}\times_{K}W_{0}$ \cite{Palais, OR}. We can then shrink $W_{0}$
such that $\mathcal{V}^{\prime}$ is contained in $\mathcal{V}$. This provides
slice coordinates $\xi^{\prime}=[(g,w)]\in G_{2}\times_{K}W_{0}$ near
$G_{2}(a)$ with respect to which $a$ corresponds to the class $[(e,0)]$. We
can thus write$\ $the $H_{a_{0}}\times K$-invariant lift $\Psi_{\varepsilon
}^{\prime}(g,w)$ of $\Psi_{\varepsilon}^{\prime}(\xi^{\prime})$ with respect
to the variables $(g,w)\in G_{2}\times W_{0}$, where the twisted action of
$H_{a_{0}}\times K$ on $G_{2}\times W_{0}$ is given by
\[
(h,k)\cdot(g,w)=(hgk^{-1},k\cdot w)\quad(h,k)\in H_{a_{0}}\times K.
\]

By $G_{2}$-equivariance of $\Psi_{0}^{\prime}(g,w)=V(\xi^{\prime})$, we have
\[
\nabla_{w}\Psi_{0}^{\prime}(g,0)=0\quad\mbox{for every}\quad g\in G_{2}%
\]
where $\nabla_{w}\Psi_{0}^{\prime}:G\times W_{0}\rightarrow W$ denotes the
projection of $\nabla\Psi_{0}^{\prime}$ to the slice $W$. In the previous
section we performed a finite-dimensional Lyapunov-Schmidt reduction and a
second Lyapunov-Schmidt to solve the singular part of $\nabla\Psi
_{\varepsilon}^{\prime}(ra_{0}, \xi^{\prime})$. Now we perform a third
Lyapunov-Schmidt reduction to express the (normal) variables $w\in W_{0}$ in
terms of the variables along the group orbit $g\in G_{2}$. For this purpose we
also need the following non-degeneracy condition on the central configuration:

\begin{definition}
\label{non degenerate cc} We say that $a$ is \defn{non-degenerate} if the only
zero eigenvalues of the Hessian $\nabla^{2}V(a)$ correspond to the
eigenvectors belonging to the tangent space $T_{a}U(d)(a)$.
\end{definition}

\begin{remark}\label{remark: non-deg}
In the case \textbf{(C1)-(C2)} we have $G_{2}=U(1)$. In the case $\mathbf{%
(C3)}$ the group $G_{2}\subset U(d)$ is lower-dimensional than the group $%
U(d)$.The central configuration $a\ $in the fixed point space of $\Gamma $
consists of points lying in the orthogonal complement $\Pi ^{\perp }$, i.e. $%
a\in E^{\prime }=\left( \Pi ^{\perp }\right) ^{n}$. The orbit $\ U(d)(a)\subset E^{n}$ intersects $\left( \Pi ^{\perp }\right) ^{n}$ in
the $G_{2}$-orbit $G_{2}(a)\subset \left( \Pi ^{\perp }\right) ^{n}$. Since $%
a\in E^{n}$ is non-degenerate, the hessian $\nabla ^{2}V(a):E^{\prime
}\rightarrow E^{\prime }$ is non-singular when restricted to a complement of
the tangent space of the group orbit $G_{2}(a)\subset E^{\prime }$. Note that we could
have considered a degenerate central configuration $a$ such that the hessian
restricted to the fixed point space of $\Gamma $, $\nabla ^{2}V(a):E^{\prime
}\rightarrow E^{\prime }$ , has only zero eigenvalues with eigenvectors
belonging to the tangent space of the orbit $T_{a}G_{2}(a)\subset E^{\prime }
$. Although, we ignore if any degenerate central configuration $a$ 
satisfies this weaker condition.
\end{remark}

Before concluding the proof of the existence of critical points for the
regular functional $\Psi_{\varepsilon}^{\prime}(g,w)$, we briefly recall some
tools of Lyusternik-Schnirelmann theory \cite{LS}. Given a compact Lie group
$G$ acting on a compact manifold $M$ and a smooth $G$-invariant function
$f:M\rightarrow\R$, the equivariant version of the Lyusternik-Schnirelmann
theorem states that the number of $G$-orbits of critical points of $f$ is
bounded below by $\mbox{Cat}_{G}(M)$ \cite{Fadell}. The latter is defined as
being the least number of \emph{$G$-categorical} open subsets required to
cover $M$. Those are the $G$-invariant open subsets which are contractible
onto a $G$-orbit by mean of a $G$-equivariant homotopy.

\begin{theorem}
\label{main category} Assume conditions $\mathbf{(A)-(B)}$ and
$\mathbf{(C1)-(C3)}$ and suppose that the central configuration $a\in E^{n}$
is non-degenerate. For each $\varepsilon\in(0,\varepsilon_{0})$ there is a
neighbourhood $\mathcal{V}^{\prime}\subset E^{\prime}$ of the orbit $G_{2}(a)$
so that the number of $H_{a_{0}}$-orbits of critical points of the reduced
potential $\Psi_{\varepsilon}^{\prime}$ defined on $\mathcal{V}^{\prime}$ is
bounded below by
\[
\mbox{Cat}_{H_{a_{0}}}(G_{2}/K).
\]
Furthermore, we have that the $H_{a_{0}}$-orbits of solutions have an element
of the form $\xi^{\prime}=g\cdot a+\mathcal{O}_{E^{\prime}}(\varepsilon)$ for
some $g\in G_{2}$.
\end{theorem}

\begin{proof}
The fact that $\nabla_{\xi^{\prime}}\mathcal{N}^{\prime}(\xi^{\prime})$ is
analytic in $\varepsilon$ and $\xi^{\prime}$ and the uniform estimate
$\Vert\nabla_{\xi^{\prime}}\mathcal{N}^{\prime}(\xi^{\prime})|_{\varepsilon
=0}\Vert\leq N^{\prime}\Vert\xi^{\prime}-a\Vert^{2}$ imply that $\nabla
_{\xi^{\prime}}\mathcal{N}^{\prime}(\xi^{\prime})$ has no linear terms at
$\xi^{\prime}=a$, i.e. $\nabla_{\xi^{\prime}}^{2}\mathcal{N}^{\prime}\left(
a\right)  |_{\varepsilon=0}=0$. Thus
\[
\nabla_{\xi^{\prime}}^{2}\Psi_{0}^{\prime}\left(  a\right)  =\nabla^{2}V(a).
\]
Under the non-degeneracy assumption, the hessian $\nabla^{2}V(a)$ is
non-singular when restricted to a complement of the tangent space of the group
orbit $G_{2}(a)$ at $a$ (see remark \ref{remark: non-deg}). Since $W$ is the orthogonal complement to the
tangent space $T_{a}G_{2}(a)$ in $E^{\prime}$ and $\nabla_{w}^{2}\Psi
_{0}^{\prime}(g,0)$ and $\nabla_{w}^{2}\Psi_{0}^{\prime}(e,0)$ are conjugated
matrices, then the inverse of $\nabla_{w}^{2}\Psi_{0}^{\prime}(g,0)$ is
bounded by $C$ for all $g\in G_{2}$ and $\varepsilon\in(0,\varepsilon_{0})$.
That is,%
\[
\left\Vert \lbrack\nabla_{w}^{2}\Psi_{0}^{\prime}(g,0)]^{-1}\right\Vert \leq
C,\qquad\forall g\in G_{2}.
\]

Fix $\varepsilon\in(0,\varepsilon_{0})$. The compactness of $G_{2}$ and an
argument based on the implicit function theorem similar to the first
Lyapunov-Schmidt reduction imply the existence of a unique map $\phi
_{\varepsilon}:G_{2}\rightarrow W_{0}$, defined for every $\varepsilon
\in(0,\varepsilon_{0})$, that solves the equation
\[
\nabla_{w}\Psi_{\varepsilon}^{\prime}(g,\phi_{\varepsilon}(g))=0,\quad g\in
G_{2}.
\]
Since $\Psi_{\varepsilon}^{\prime}(g,w)$ is $H_{a_{0}}\times K$-invariant, the
uniqueness of the map $\phi_{\varepsilon}(g)$ implies that $\phi_{\varepsilon
}(hgk^{-1})=k\cdot\phi_{\varepsilon}(g)$, i.e. $\phi_{\varepsilon}(g)$ is
$H_{a_{0}}$-invariant and $K$-equivariant. In particular, finding the
solutions of $\nabla\Psi_{\varepsilon}^{\prime}(\xi^{\prime})=0$ amounts to
find the critical points of $\Psi_{\varepsilon}^{\prime}(g,\phi_{\varepsilon
}(g)):G_{2}\rightarrow\mathbb{R}$ which descends to an $H_{a_{0}}$-invariant
function on $G_{2}/K$ - which is compact. By the equivariant version of the
Lyusternik-Schnirelmann theorem, the number of $H_{a_{0}}$-orbits of critical
points of $\Psi_{\varepsilon}^{\prime}:G_{2}/K\rightarrow\mathbb{R}$ is
bounded below by $\mbox{Cat}_{H_{a_{0}}}(G_{2}/K)$. Finally, using a similar
argument to Lemma \ref{lemma: estimate phi} it is possible to show that
$\Vert\phi_{\varepsilon}(g)\Vert\leq N_{3}\varepsilon$ for every $g\in G_{2}$.
Thus these $H_{a_{0}}$-orbits have an element of the form $\xi^{\prime}=g\cdot
a+\mathcal{O}_{E^{\prime}}(\varepsilon)$ for some $g\in G_{2}$.
\end{proof}

\section{Solutions of the $N$-body problem}

We now work out the solutions that we obtain for the $N=(n+1)$-body problem%

\begin{equation}
m_{\ell}\ddot{q}_{\ell}=-\sum_{k\neq\ell}m_{\ell}m_{k}\frac{q_{\ell}-q_{k}%
}{\Vert q_{\ell}-q_{k}\Vert^{\alpha+1}},\qquad\ell=0,\dots,n \label{NBP}%
\end{equation}
according to the three cases $\mathbf{(C1)}$-$\mathbf{(C2)}$-$\mathbf{(C3)}$
that we discussed earlier. The solutions are now written in components
\[
q(t)=(q_{0}(t),q_{1}(t),\dots,q_{n}(t))\in E^{N}.
\]

\subsection{Solutions in the plane (C1)-(C2)}

If $E$ is two dimensional, we set $\mathcal{J}=J.$ In this case we obtain
solutions that in some particular cases correspond to braids. In this case,
$G=U(1)\times U(1)$ and $H=\widetilde{U(1)}$ is diagonal in $G$. The orbit of
$x_{a}$ is $G(x_{a})=G_{1}(a_{0})\times G_{2}(a)$ whose two factors are
isomorphic to a circle $U(1)$. Furthermore, the groups $H_{a_{0}}$ and $K$ are
trivial, the orbit is $G_{2}(a)=\mathbb{S}^{1}$ and
\[
\mbox{Cat}(G_{2}/K)=2.
\]
By Theorem \ref{main category}, the regular functional $\Psi_{\varepsilon
}^{\prime}(\xi^{\prime})$ has at least two critical points near $G_{2}(a)$. We
can identify the critical points of $\Psi_{\varepsilon}^{\prime}(g,\phi
_{\varepsilon}(g))$ by an element of the form $g=e^{\vartheta J}\in G_{2}=U(1)$ for some
$\vartheta\in\lbrack0,2\pi]$. Then for the critical points of $\Psi(\xi)$ we
have $\xi=\left(  a_{0},e^{i\vartheta J}a\right)  +\mathcal{O}_{X_{0}^{\Gamma
}}(\varepsilon)$. Therefore, the critical points of $\mathcal{A}(\xi+\eta)$
are given by
\[
u=\xi+\eta=(a_{0},e^{i\vartheta J}a)+\mathcal{O}_{X^{\Gamma}}(\varepsilon
)\text{,}%
\]
where $\mathcal{O}_{X^{\Gamma}}(\varepsilon)$ is a function in $X^{\Gamma}$
such that $\left\Vert \mathcal{O}_{X^{\Gamma}}(\varepsilon)\right\Vert
_{X}\leq c\varepsilon$ for some constant $c$. Then we have,

\begin{theorem}
\label{main result} Suppose that $\dim(E)=2$ and the conditions
$\mathbf{(A)-(B)}$ are satisfied. Let $a\in E^{n}$ be a central configuration,
satisfying the conditions $\mathbf{(C1)}$ or $\mathbf{(C2)}$, and such that
$\nabla^{2}V(a)$ has kernel of real dimension $1$. Then the following occurs:

\begin{itemize}
\item[$\mathbf{(C1)}$] If $\alpha\neq2$, then for every $\varepsilon
\in(0,\varepsilon_{0})$, there are at least two solutions $q(t)$ of
\eqref{NBP} with components of the form
\begin{align}
q_{0}(t)  &  =\exp\left(  tJ\right)  u_{1}(\nu t)-m_{1}\varepsilon\exp(t\omega
J)u_{0}(\nu t)\\
q_{1}(t)  &  =\exp\left(  tJ\right)  u_{1}(\nu t)+m_{0}\varepsilon\exp(t\omega
J)u_{0}(\nu t)\nonumber\\
q_{\ell}(t)  &  =\exp\left(  tJ\right)  u_{\ell}(\nu t),\qquad\ell
=2,...,n~,\nonumber
\end{align}
where $\omega=\pm\varepsilon^{-(\alpha+1)/2}$, $\nu=\omega-1$, $u_{0}%
(s)=a_{0}+\mathcal{O}_{X^{\Gamma}}(\varepsilon)$ and $u_{\ell}(s)=e^{\vartheta
J}a_{\ell}+\mathcal{O}_{X}(\varepsilon)$ for some phase $\vartheta\in
\lbrack0,2\pi]$. The case $\omega>0$ corresponds to a prograde rotation of the
pair and $\omega<0$ to a retrograde rotation.

\item[$\mathbf{(C2)}$] If $\alpha=2$, the same result holds with the addition
that $u_{0}(s)$ is $2\pi/n$-periodic and%
\begin{equation}
u_{\ell}(s)=\exp(-\theta J)u_{\sigma(\ell)}(s+\theta),\qquad\ell=1,\dots,n,
\label{sym}%
\end{equation}
where $(\theta,\sigma)$ is the generator of the discrete symmetry group
$\Gamma$ defined in Section \ref{section: symmetries}.
\end{itemize}
\end{theorem}

For such solutions, the bodies $\ell=0,1$ rotate in a circular Kepler orbit
whose center of mass follows the position determined by a body in a rigid
motion of $n$ bodies. If $\varepsilon\in(0,\varepsilon_{0})$ is such that
$\omega\in\mathbb{Q}$, then $\nu=1-\omega\in\mathbb{Q}$ and the solution is
periodic. Otherwise the solution $q(t)$ is quasi-periodic. Furthermore, if the
frequency $\omega=\pm p/q$ is rational, then $\nu=(\pm p-q)/q$ is rational and
the functions $u_{j}(\nu t)$ and $e^{\omega tJ}$ are $2\pi q$-periodic.
Therefore, the solutions $q(t)$ are $2\pi q$-periodic.

\begin{corollary}
[Braid solutions]\label{main braids} Suppose that $\dim(E)=2$ and the
conditions $\mathbf{(A)-(B)}$ are satisfied. Let $a\in E^{n}$ be a central
configuration, satisfying the conditions $\mathbf{(C1)}$ or $\mathbf{(C2)}$,
and such that $\nabla^{2}V(a)$ has kernel of real dimension $1$. Fix an
integer $q\in\Z\setminus\{0\}$. Set $\varepsilon=\left(  p/q\right)
^{-2/(\alpha+1)}$, where $p$ is relatively prime to $q$. Then there is $p_{0}$
such that, for each $p>p_{0}$, there are at least two solutions $q(t)$ of
\eqref{NBP} with components of the form
\begin{align*}
q_{0}(t)  &  =\exp(\left(  t+\vartheta\right)  J)a_{1}-m_{1}\varepsilon
\exp(\pm(pt/q)J)a_{0}+\mathcal{O}(\varepsilon),\\
q_{1}(t)  &  =\exp(\left(  t+\vartheta\right)  J)a_{1}+m_{0}\varepsilon
\exp(\pm(pt/q)J)a_{0}+\mathcal{O}(\varepsilon),\\
q_{\ell}(t)  &  =\exp(\left(  t+\vartheta\right)  J)a_{\ell}+\mathcal{O}%
(\varepsilon),\qquad\ell=2,...,n~.
\end{align*}
where $\vartheta$ represents a phase, and $\mathcal{O}(\varepsilon)$ is a
$2\pi q$-periodic function of order $\varepsilon$.
\end{corollary}

In these solutions the bodies $\ell=0,1$ wind around their center of mass $p$
times in the period $2\pi q$, while the center of mass of the bodies
$\ell=0,1$ and the bodies $\ell=2,...,n$ wind around the origin $q$ times. The
case $\omega=p/q$ corresponds to a prograde rotation of the pair and
$\omega=-p/q$ to a retrograde rotation.

\subsection{Examples of solutions satisfying conditions (C2)}

\label{examples} Given that we need the symmetric conditions (C2a)-(C2b) in
the gravitational case, we now present examples of configurations that we can
braid: the Maxwell configuration and configurations symmetric through the
origin. For each case, we find a symmetry $\sigma\in S_{n}$ that allows to
deal with the resonances.

\begin{itemize}
\item \textbf{Maxwell configuration. }The Maxwell configuration is proposed by
Maxwell as a model of Saturn and its ring. This central configuration consists
of a polygonal configuration of unitary masses with a central body of
different mass $\mu$. The central body is at the origin $a_{1}=0$ with mass
$M_{1}=\mu$. The other bodies have masses $M_{\ell}=1$ and coordinates
\[
a_{\ell}=\left(  \mu+S_{n-1}\right)  ^{3/2}e^{J\ell\theta}%
\begin{bmatrix}
1\\
0
\end{bmatrix}
,\qquad\theta=\frac{2\pi}{n-1}%
\]
for $\ell=1,\dots,n-1$, where
\[
S_{n-1}=\frac{1}{4}\sum_{\ell=1}^{n-1}\frac{1}{\sin(\frac{\pi\ell}{n-1}%
)}\text{.}%
\]
(see \cite{GaIz13} for details). We consider the discrete symmetry generated
by $(\theta,\sigma)$, where $\sigma=(2\dots n)\in S_{n}$ is such that
$\sigma^{n-1}=(1)$. We only need to verify conditions (C2a)-(C2b). The masses
satisfy condition (C2a) because $\sigma(1)=1$ and $M_{\ell}=1$ for
$\ell=2,\dots,n$. The positions satisfy condition (C2b) because $\sigma(1)=1$
with $a_{n}=0\ $and $a_{\sigma(\ell)}=\exp(\theta J)a_{\ell}$ for
$\ell=2,\dots,n$.
\end{itemize}

\begin{center}
\begin{figure}[h]
\begin{center}
\begin{tikzpicture}
\tkzDefPoint(0,0){O}\tkzDefPoint(2,0){a_2}
\tkzDefPoint(-1,0){P}\tkzDefPoint(0,0){a_7}
\tkzDefPointsBy[rotation=center O angle 360/6](a_2,a_3,a_4,a_5,a_6){a_3,a_4,a_5,a_6,a_1}
\tkzLabelPoints[right](a_2)
\tkzLabelPoints[above right](a_3)
\tkzLabelPoints[below, yshift=-0.2cm](a_7)
\tkzLabelPoints[above left](a_4)
\tkzLabelPoints[left](a_5)
\tkzLabelPoints[below left](a_6)
\tkzLabelPoints[below right](a_1)
\tkzDrawPolygon[color=gray, dashed](a_2,a_3,a_4,a_5,a_6,a_1)
\tkzDrawPoints[fill =black,size=10,color=black](a_2,a_3,a_4,a_5,a_6,a_1)
\tkzDrawPoints[fill =iris,size=20,color=iris](a_7)
\end{tikzpicture}
\end{center}
\caption{Maxwell configuration for seven bodies.}%
\end{figure}
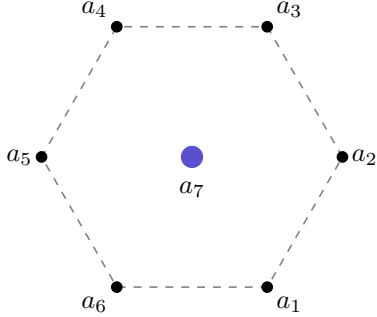
\end{center}

\begin{itemize}
\item \textbf{Symmetric configuration with respect to the origin.} In this
case we assume that $\theta=\pi$ and that there is an involution $\sigma\in
S_{n}$ such that $\sigma^{2}=(1)$ and $\sigma(1)=1$. That is, the central
configuration $a$ and its associated masses $M_{\ell}$ need to be invariant
under the involution $\sigma$. Explicitly we require
\[
M_{\ell}=M_{\sigma(\ell)},\qquad a_{\sigma(\ell)}=-a_{\ell},
\]
for $\ell=1,\dots,n$, i.e. $a_{1}=0$. This class of central configuration are
symmetric with respect to the origin.

\begin{figure}[h]
\begin{center}
\begin{tikzpicture}[fill opacity=1, scale=0.8]
\tkzDefPoint(0,0){O}\tkzDefPoint(1.5,-1.5){a_1}
\tkzDefPointsBy[rotation=center O angle 360/4](a_1,a_2,a_3,a_4){a_2,a_3,a_4}
\tkzDefPoint(-1,0){P}\tkzDefPoint(0,0){a_5}
\tkzDrawPolygon[color=gray, dashed](a_1,a_2,a_3,a_4)
\tkzDefPoint(2,0){b_1}
\tkzDefPointsBy[rotation=center O angle 360/4](b_1,b_2,b_3,b_4){b_2,b_3,b_4}
\tkzDrawPolygon[color=gray, dashed, fill opacity=0.2](b_1,b_2,b_3,b_4)
\tkzDefPoint(3,0){c_1}
\tkzDefPoint(-3,0){c_2}
\tkzDefPoint(0,3){d_1}
\tkzDefPoint(0,-3){d_2}
\tkzDefPoint(2.5,0){A}
\tkzDefPoint(-2.5,0){B}
\tkzDrawLine[color=gray, dashed](A,B)
\tkzDefPoint(0,2.5){C}
\tkzDefPoint(0,-2.5){D}
\tkzDrawLine[color=gray, dashed](C,D)
\tkzDrawPoints[fill =black,size=10,color=black](a_1,a_3)
\tkzDrawPoints[fill =red,size=16,color=red](a_2,a_4)
\tkzDrawPoints[fill =iris,size=20,color=iris](a_5)
\tkzDrawPoints[fill =red, size=10,color=red](b_1,b_3)
\tkzDrawPoints[fill =asparagus, size=14,color=asparagus](b_2,b_4)
\tkzDrawPoints[size=15,fill = blue!50!black](c_1,c_2)
\tkzDrawPoints[size=15,fill = blue!50!black](d_1,d_2)
\end{tikzpicture}
\end{center}
\caption{A central configuration with $D_{2}$ symmetry (see \cite{Mo2} for the
existence of such configurations). }%
\end{figure}
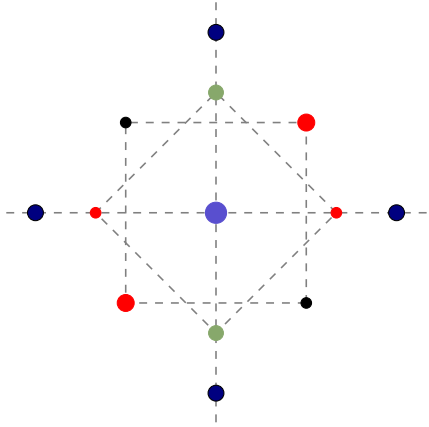
\end{itemize}

\subsection{Solutions in more dimensions (C3)}

For $d\geq2$ the symmetry group is $G=G_{1}\times G_{2}$ where $G_{1}%
=G_{2}=U(1)\times U(d-1)$. Since $a_{0}$ is in the plane $E_{0}=\Pi
:=\{(x,y,0,...,0)\}$, the group orbit of $x_{a}$ is identified with
\[
G(x_{a})=G_{1}(a_{0})\times G_{2}(a)
\]
where $G_{1}(a_{0})=\mathbb{S}^{1}$. Note that, in this case, $H_{a_{0}%
}=\left\{  e\right\}  \times U(d-1). $ By assumption $\mathbf{(C3)}$, the
central configuration $a$ lies in the subspace orthogonal to the plane,
$E_{1}=\Pi^{\perp}$, then $G_{2}(a)=\mathbb{S}^{2d-3}$. It follows that
\[
\mbox{Cat}_{H_{a_{0}}}\left(  G_{2}(a)\right)  =\mbox{Cat}\left(  pt\right)
=1.
\]

We can identify the critical orbit of $\Psi_{\varepsilon}^{\prime}(g,\phi
_{1}(g))$ with any element $g\in G_{2}$. Therefore, the critical point of
$\mathcal{A}(\xi+\eta)$ is given by
\[
u=\xi+\eta=(a_{0},ga)+\mathcal{O}_{X^{\Gamma}}(\varepsilon)\text{,}%
\]
where $\mathcal{O}_{X^{\Gamma}}(\varepsilon)$ is a function in $X^{\Gamma}$
such that $\left\Vert \mathcal{O}_{X^{\Gamma}}(\varepsilon)\right\Vert
_{X}\leq c\varepsilon$ for some constant $c$.

If the central configuration $a$ is non-degenerate, then the Hessian of the
amended potential $V$ is invertible in the orthogonal complement to the
tangent space to the orbit $G_{2}(a)$ in the fixed point space of $\Gamma$.

\begin{theorem}
\label{main result copy(1)} Assume conditions $\mathbf{(A)}$-$\mathbf{(B)}$
and $\mathbf{(C3)}$. Suppose that $a\in E^{n}$ is not-degenerate. Then, for
every $\varepsilon\in(0,\varepsilon_{0})$, the $N=n+1$-body problem has at
least one solutions $q(t)$ of the form
\begin{align}
q_{0}(t)  &  =\exp\left(  t\mathcal{J}\right)  u_{1}(\nu t)-m_{1}%
\varepsilon\exp(t\omega\mathcal{J})u_{0}(\nu t)\label{qqq}\\
q_{1}(t)  &  =\exp\left(  t\mathcal{J}\right)  u_{1}(\nu t)+m_{0}%
\varepsilon\exp(t\omega\mathcal{J})u_{0}(\nu t)\nonumber\\
q_{\ell}(t)  &  =\exp\left(  t\mathcal{J}\right)  u_{\ell}(\nu t),\qquad
\ell=2,...,n~,\nonumber
\end{align}
where $\omega=\pm\varepsilon^{-(\alpha+1)/2}$, $\nu=\omega-1$, $u_{0}%
(s)=a_{0}+\mathcal{O}_{X^{\Gamma}}(\varepsilon)$ and $u_{\ell}(s)=ga_{\ell
}+\mathcal{O}_{X^{\Gamma}}(\varepsilon)$ with $g\in\left\{  e\right\}  \times
U(d-1)\subset U(d)$. Furthermore, in this case $u_{0}(s)$ and $u_{\ell}(s)$
have the symmetries $u_{0}(s)=-\mathcal{R}u_{0}(s+\pi)$ and $u_{\ell
}(s)=\mathcal{R}u_{\ell}(s+\pi)$.
\end{theorem}

\vspace{0.5cm}
\begin{minipage}[t]{7cm}
MF:  {marine.fontaine.math@gmail.com}\\
{\tt Departement Wiskunde-Informatica \\
Universiteit Antwerpen \\
2020 Antwerpen, BE.}

\end{minipage}
\hfill
\begin{minipage}[t]{7cm}
CGA:  {cgazpe@mym.iimas.unam.mx}\\
{\tt Depto. Matem\'{a}ticas y Mec\'{a}nica IIMAS \\
Universidad Nacional Aut\'onoma de M\'exico,
Apdo. Postal 20-726, Ciudad de M\'exico, MX.}
\end{minipage}
\end{document}